\documentclass[a4paper,UKenglish,cleveref, autoref, thm-restate,numberwithinsect]{lipics-v2019}
\usepackage[T1]{fontenc}
\usepackage[utf8]{inputenc}
\usepackage{tabularx}
\usepackage{todonotes}

\hypersetup{hypertexnames=false,colorlinks=true,allcolors=blue}

\newcommand{\cC}{\mathcal{C}}
\newcommand{\cD}{\mathcal{D}}

\newcommand{\cI}{\mathcal{I}}
\newcommand{\cJ}{\mathcal{J}}

\newcommand{\cM}{\mathcal{M}}
\newcommand{\cN}{\mathcal{N}}
\newcommand{\cQ}{\mathcal{Q}}

\newcommand{\Av}{\text{Av}}

\DeclareMathOperator{\Grid}{Grid}
\newcommand{\St}{\text{St}}

\DeclareMathOperator{\intv}{int}
\newcommand{\intx}{\intv_x}
\newcommand{\inty}{\intv_y}
\newcommand{\gw}{\mathrm{gw}}

\newcommand{\pw}{\mathrm{pw}}
\newcommand{\tw}{\mathrm{tw}}

\newcommand{\PPM}{\textsc{PPM}}
\newcommand{\PPPM}[1]{\textsc{#1-Pattern PPM}}

\theoremstyle{claimstyle}
\newtheorem{observation}[theorem]{Observation}

\theoremstyle{plain}
\newtheorem{problem}{Open problem}

\newcommand{\problemdef}[3] {
\vspace{\parskip}
\noindent
\begin{tabularx}{\linewidth}{| r X|}
\hline
\multicolumn{2}{|l|}{
\rule{0pt}{1.2em} \textsc{#1}} \\
 {\itshape Input:} & #2 \\ 
 {\itshape Question:} & #3\\
 \hline
\end{tabularx}
}

\newcommand{\Inc}{\mathrel{
		\begin{tikzpicture}[line cap=round, line join=round]
			\draw (0ex, 0ex) rectangle (1.5ex, 1.5ex)
			(0ex, 0ex) -- (1.5ex, 1.5ex);
		\end{tikzpicture}
}}
\newcommand{\Dec}{\mathrel{
		\begin{tikzpicture}[line cap=round, line join=round]
			\draw (0ex, 0ex) rectangle (1.5ex, 1.5ex)
			(1.5ex, 0ex) -- (0ex, 1.5ex);
		\end{tikzpicture}
}}

\graphicspath{{./graphics/}}%helpful if your graphic files are in another directory

\title{A Complexity Dichotomy for Permutation Pattern Matching on Grid Classes}

\author{Vít Jelínek}
{Computer Science Institute, Charles University, Prague, Czechia}
{jelinek@iuuk.mff.cuni.cz}{https://orcid.org/0000-0003-4831-4079}{Supported by project 18-19158S of 
the Czech Science 
Foundation.}

\author{Michal Opler}
{Computer Science Institute, Charles University, Prague, Czechia}
{opler@iuuk.mff.cuni.cz}{https://orcid.org/0000-0002-4389-5807}{Supported by project 18-19158S of 
the Czech Science 
Foundation.}

\author{Jakub Pekárek}
{Department of Applied Mathematics, Charles University, Prague, Czechia}
{pekarej@iuuk.mff.cuni.cz}{https://orcid.org/0000-0002-5409-3930}{}

\authorrunning{V. Jelínek, M. Opler and J. Pekárek} %TODO mandatory. First: Use abbreviated 
\Copyright{Vít Jelínek, Michal Opler and Jakub Pekárek} %TODO mandatory, please use full first 
 %TODO mandatory; please add comma-separated 
\category{} %optional, e.g. invited paper

\relatedversion{} %optional, e.g. full version hosted on arXiv, HAL, or other repository/website
\funding{Supported by the GAUK project 1766318.}%optional, to capture a funding statement, 
\nolinenumbers %uncomment to disable line numbering

\hideLIPIcs  %uncomment to remove references to LIPIcs series (logo, DOI, ...), e.g. when 
\EventEditors{Javier Esparza and Daniel Kr{\'a}l'}
\EventNoEds{2}
\EventLongTitle{45th International Symposium on Mathematical Foundations of Computer Science (MFCS 
2020)}
\EventShortTitle{MFCS 2020}
\EventAcronym{MFCS}
\EventYear{2020}
\EventDate{August 24--28, 2020}
\EventLocation{Prague, Czech Republic}
\EventLogo{}
\SeriesVolume{170}
\ArticleNo{44}
\begin{document}
\maketitle

\begin{abstract}
\textsc{Permutation Pattern Matching (PPM)} is the problem of deciding for a given pair of 
permutations $\pi$ and $\tau$ whether the pattern $\pi$ is contained in the text $\tau$. Bose, Buss 
and Lubiw showed that PPM is NP-complete. In view of this result, it is natural to ask 
how the situation changes when we restrict the pattern $\pi$ to a fixed permutation class $\cC$;
this is known as the \textsc{$\cC$-Pattern PPM} problem. There have been several results in this 
direction, namely the work of Jelínek and Kynčl who completely resolved the hardness of 
\textsc{$\cC$-Pattern PPM} when $\cC$ is taken to be the class of $\sigma$-avoiding permutations 
for some~$\sigma$.

\emph{Grid classes} are special kind of permutation classes, consisting of permutations admitting a 
grid-like decomposition into simpler building blocks. Of particular interest are the so-called 
\emph{monotone grid classes}, in which each building block is a monotone sequence. Recently, it has 
been discovered that grid classes, especially the monotone ones, play a fundamental role in the 
understanding of the structure of general permutation classes. This motivates us to study the 
hardness of \textsc{$\cC$-Pattern PPM} for a (monotone) grid class~$\cC$.

We provide a complexity dichotomy for \textsc{$\cC$-Pattern PPM} when $\cC$ is taken to be a 
monotone grid class. Specifically, we show that the problem is polynomial-time solvable if a 
certain graph associated with $\cC$, called the cell graph, is a forest, and it is NP-complete 
otherwise. We further generalize our results to grid classes whose blocks belong to classes of 
bounded grid-width. We show that the \textsc{$\cC$-Pattern PPM} for such a grid class $\cC$ is 
polynomial-time solvable if the cell graph of $\cC$ avoids a cycle or a certain special type of 
path, and it is NP-complete otherwise.
\end{abstract}

\section{Introduction}
A \emph{permutation} is a sequence $\pi=\pi_1,\pi_2,\dotsc,\pi_n$ in which each number from the set 
$[n]=\{1,2,\dotsc,n\}$ appears exactly once. We then say that a permutation 
$\pi=\pi_1,\dotsc,\pi_n$ \emph{contains} a permutation $\sigma=\sigma_1,\dotsc,\sigma_k$, if $\pi$ 
has a subsequence of length $k$ whose elements have the same relative order as the elements 
of~$\sigma$ (see Section~\ref{sec:defs} for a more formal definition). If $\pi$ does not contain 
$\sigma$, we say that $\pi$ \emph{avoids} $\sigma$.

An essential algorithmic problem involving permutations is \textsc{Permutation Pattern Matching 
(PPM)}: Given a permutation $\pi$ (`pattern') of size $k$ and $\tau$ (`text') of size $n$, does 
$\tau$ contain $\pi$?
Bose, Buss and Lubiw~\cite{Bose1998} have shown that \PPM\ is NP-complete. This result has motivated 
a study of various variants of \PPM, in particular to obtain the best possible runtime dependence 
on $k$. Guillemot and Marx~\cite{Guillemot2014} provided the break-through result in this direction 
by establishing the fixed-parameter tractability of 
\PPM\ in terms of the pattern length with algorithm running in $2^{O(k^2 \log 
k)}\cdot n$ time. The first phase of their algorithm finds a suitable decomposition that relies on 
the proof of Stanley-Wilf conjecture given by Marcus and Tardos~\cite{Marcus2004}. Subsequently, 
Fox~\cite{Fox2013} refined the results by Marcus and Tardos and thereby reduced the complexity of 
the algorithm to $2^{O(k^2)} \cdot n$.

%For long patterns, \PPM can be solved using a straightforward brute-force approach in $O(2^n)$. 
%The 
%first improvement in this area was obtained by Bruner and Lackner~\cite{Bruner2016} whose 
%algorithm runs in time $O(1.79^n)$. Recently, this was improved by Berendsohn, Kozma and 
%Marx~\cite{Berendsohn2019} to a running time $O(1.618^n)$.

Several algorithms also exist whose runtimes depend on different parameters than the length of 
$\pi$. Bruner and Lackner~\cite{Bruner2016} described an algorithm for \PPM\ with run time
$O(1.79^{\text{run}(\tau)}\cdot kn)$ where $\text{run}(\tau)$ is the number of consecutive monotone 
sequences needed to obtain $\tau$ via concatenation. Ahal and Rabinovich~\cite{Ahal2000} designed 
an algorithm for \PPM\ that runs in time $n^{O(\tw(G_\pi))}$ where $G_\pi$ is a certain graph 
associated to the pattern $\pi$ and $\tw(G_\pi)$ denotes the treewidth of $G_\pi$. Later Jelínek, 
Opler and Valtr~\cite{Jelinek2018} introduced a related parameter $\gw(\pi)$, called the 
\emph{grid-width} of $\pi$, and showed that $\gw(\pi)$ is equivalent to $\tw(G_\pi)$ up to a 
constant and thereby implying an algorithm for \PPM\ running in time $O(n^{O(\gw(\pi))})$.

Another approach to tackling the hardness of \PPM\ is to restrict the choice of pattern to a 
particular permutation class $\cC$, where a \emph{permutation class} is a set of permutations $\cC$ 
such that for every $\sigma \in \cC$, every permutation contained in $\sigma$ belongs to $\cC$ as 
well.

\problemdef{$\cC$-Pattern Permutation Pattern Matching ($\cC$-Pattern PPM)}
{A pattern $\pi\in\cC$ of size $k$ and a permutation $\tau$ of size $n$.}
{Does $\tau$ contain $\pi$?}

For a permutation $\sigma$, we let $\Av(\sigma)$ denote the class of permutations 
avoiding~$\sigma$. 
Notice that \PPPM{$\Av(21)$} simply reduces to finding the longest increasing subsequence of 
$\tau$, which is a well-known problem and can be solved in time $O(n \log\log 
n)$~\cite{Makinen2001}. 
Bose, Buss and Lubiw~\cite{Bose1998} showed that \PPPM{$\cC$} can be solved in polynomial time if 
$\cC$ is the class of the so-called separable permutations. Further improvements 
were given by Ibarra~\cite{Ibarra1997}, Albert et al.~\cite{Albert2001} and by Yugandhar and 
Saxena~\cite{Yugandhar2005}.
Recently, Jelínek and Kynčl~\cite{Jelinek2017} completely resolved the hardness of \PPPM{$\cC$} for 
classes avoiding a single pattern. They proved that \PPPM{$\Av(\alpha)$} is polynomial-time 
solvable for $\alpha \in \{1,12,21,132,213,231,312\}$ and NP-complete otherwise. 
%Note that they 
%have not designed any novel algorithm, rather they  proved that for $\alpha \in 
%\{1,12,32,132,213,231,312\}$ and for every pattern $\pi \in Av(\alpha)$ the parameter $\tw(G_\pi)$ 
%is bounded. 
%The positive part of their result then follows from the aforementioned algorithm of 
%Ahal and Rabinovich with running time $O(n^{O(\tw(G_\pi))})$.

Lately, a new type of permutation classes has been gaining a lot of attention. The \emph{grid 
class} of a matrix $\cM$ whose entries are permutation classes, denoted by $\Grid(\cM)$, is a class 
of permutations admitting a grid-like decomposition into blocks that belong to the classes 
$\cM_{i,j}$. If moreover $\cM$ contains only $\Av(21), \Av(12)$ and $\emptyset$ we say that 
$\Grid(\cM)$ is a \emph{monotone grid class}. To each matrix $\cM$ we also associate a graph 
$G_\cM$, called the \emph{cell graph} of $\cM$. We postpone their full definitions to 
Section~\ref{sec:defs}.

Monotone grid classes were introduced partly by Atkinson, Murphy and Ru\v{s}kuc~\cite{Atkinson2002} 
and in full by Murphy and Vatter~\cite{Murphy2002} who showed that a monotone grid class 
$\Grid(\cM)$ is partially well-ordered if and only if $G_\cM$ is a forest. 
Brignall~\cite{Brignall2012} later extended their results to a large portion of general grid 
classes. General grid classes themselves were introduced by Vatter~\cite{Vatter2011} in his paper 
investigating the growth rates of permutation classes. Since then the grid classes
played a central role in most subsequent works on growth rates of permutation
classes~\cite{Albert2013,Vatter2019}. Bevan~\cite{Bevan2014,Bevan2015} tied 
the growth rates of grid classes to algebraic graph theory.

Given the prominent role of grid classes in recent developments of the permutation pattern 
research, it is only natural to investigate the hardness of searching patterns that belong to a 
grid class. Neou, Rizzi and Vialette~\cite{Neou2016} designed a polynomial-time algorithm solving 
\PPPM{$\cC$} when $\cC$ is the class of the so-called wedge permutations, which can be also 
described as a monotone grid class. Consequently, Neou~\cite{Neou2017} asks at the end of his 
thesis about the hardness of \PPPM{$\Grid(\cM)$} for a monotone grid class $\Grid(\cM)$.

\subparagraph*{Our results.} 
In Section~\ref{sec:monotone}, we answer the question of Neou by proving that for a monotone grid 
class $\Grid(\cM)$, the problem \PPPM{$\Grid(\cM)$} is polynomial-time solvable  if the cell graph 
$G_\cM$ is a forest, NP-complete otherwise.

In Section~\ref{sec:general}, we further extend our results to matrices whose every entry is a 
permutation class of bounded grid-width. We prove that for such grid classes, \PPPM{$\Grid(\cM)$} 
is polynomial-time solvable if $G_\cM$ is a forest which does not contain a certain type of path, 
and NP-complete otherwise.

\section{Preliminaries}\label{sec:defs}
A \emph{permutation of length $n$} is a sequence in which each element of the set $[n] = \lbrace 1, 
2, \dots, n\rbrace$ appears exactly once.  When writing out short permutations explicitly, we shall 
omit all punctuation and write, e.g., $15342$ for the permutation $1,5,3,4,2$. The 
\emph{permutation diagram} of $\pi$ is the set of points $\{(i,\pi_i);\;i\in[n]\}$ in the plane. 
Note that we use Cartesian coordinates, that is, the first row of the diagram is at the bottom. We 
blur the distinction between permutations and their permutation diagrams, e.g., we shall refer to 
`the point set $\pi$' rather than `the point set of the diagram of the permutation~$\pi$'.

For a point $p$ in the plane, we denote its first coordinate as $p.x$, and its second coordinate as 
$p.y$. A subset $S$ of a permutation diagram is \emph{isomorphic} to a subset $R$ of a permutation 
diagram if there is a bijection $f\colon R \to S$ such that for any pair of points $p \neq q$ of 
$R$ we 
have $f(p).x < f(q).x$ if and only if $p.x < q.x$, and $f(p).y < f(p).y$ if and only if $p.y < 
q.y$. A permutation $\pi$ of length $n$ \emph{contains} a permutation $\sigma$ of length $k$, if 
there is a subset of $\pi$ isomorphic to the permutation diagram of $\sigma$. Such a subset is then 
an \emph{occurrence} (or a \emph{copy}) of $\sigma$ in~$\pi$. If $\pi$ does not contain $\sigma$, 
we say that $\pi$ \emph{avoids}~$\sigma$.

A \emph{permutation class} is a set $\cC$ of permutations with the property that if $\pi$ is in 
$\cC$, then all the permutations contained in $\pi$ are in $\cC$ as well. For a permutation 
$\sigma$, we let $\Av(\sigma)$ denote the class of $\sigma$-avoiding permutations. We shall 
sometimes use the symbols $\Inc$ and $\Dec$  as short-hands 
for the class of increasing permutations $\Av(21)$ and the class of decreasing permutations 
$\Av(12)$.

For a permutation $\pi$ of length $n$ the \emph{reverse} of $\pi$ is the permutation 
$\pi_n,\pi_{n-1},\dotsc,\pi_1$, the \emph{complement} of $\pi$ is the permutation 
$n+1-\pi_1,n+1-\pi_2,\dotsc,n+1-\pi_n$, and the \emph{inverse} of $\pi$ is the permutation 
$\sigma=\sigma_1,\dotsc,\sigma_n$ satisfying $\pi_i=j\iff\sigma_j=i$. We let $\pi^r$, $\pi^c$ and 
$\pi^{-1}$ denote the reverse, complement and inverse of $\pi$, respectively. For a 
permutation class $\cC$, we let $\cC^r$ denote the set $\{\pi^r;\; \pi\in\cC\}$, and similarly for 
$\cC^c$ and $\cC^{-1}$. Note that $\cC^r$, $\cC^c$ and $\cC^{-1}$ are again permutation classes.

A permutation $\pi$ of length $n$ is a \emph{horizontal alternation} if all the even 
entries of $\pi$ precede all the odd entries of $\pi$, i.e., there are no indices $i < j$ such that 
$\pi_i$ is odd and $\pi_j$ is even. A permutation $\pi$ is a \emph{vertical alternation} if 
$\pi^{-1}$ is a horizontal alternation.

\subparagraph*{Griddings and grid classes.} For two sets $A$ and $B$ of numbers, we write $A<B$ if 
every element of $A$ is smaller than any element 
of~$B$. In particular, both $\emptyset<A$ and $A<\emptyset$ holds for any~$A$.

A $k \times \ell$ matrix $\cM$ whose entries are permutation classes is called a \emph{gridding 
matrix}. Moreover, if the entries of $\cM$ belong to the set $\{\Inc,\Dec,\emptyset\}$ then we say 
that $\cM$ is a \emph{monotone gridding matrix}. A \emph{$k \times \ell$-gridding} of a permutation 
$\pi$ of length $n$ are two sequences of (possibly empty) disjoint integer intervals 
$I_1<I_2<\dotsb<I_k$ and $J_1<J_2<\dotsb<J_{\ell}$ such that  both $\bigcup_{i=1}^k I_i$ and 
$\bigcup_{j=1}^\ell J_j$ are equal to $[n]$. We call the set $I_i\times J_j$ an \emph{$(i,j)$-cell} 
of $\pi$. An \emph{$\cM$-gridding} of a permutation $\pi$ is a $k \times 
\ell$-gridding such that the restriction of $\pi$ to the $(i,j)$-cell is isomorphic to 
a permutation from the class~$\cM_{i,j}$. If $\pi$ possesses an $\cM$-gridding, then $\pi$ is said 
to 
be \emph{$\cM$-griddable} and $\pi$ equipped with a fixed $\cM$-gridding is called an 
\emph{$\cM$-gridded permutation}. We let $\Grid(\cM)$ be the class of $\cM$-griddable 
permutations. Note that for consistency with our Cartesian numbering convention, we number the rows 
of a matrix from bottom to top.

The \emph{cell graph} of the gridding matrix $\cM$, denoted $G_\cM$, is the graph whose vertices 
are the cells of $\cM$ that contain an infinite class, with two vertices being adjacent if they 
share a row or a column of $\cM$ and all cells between them are finite or empty. A \emph{proper 
turning path} in $G_\cM$ is a path $P$ such that no three consecutive cells of $P$ share the same 
row or column. See Figure~\ref{fig:grid-class}.

\begin{figure}
\centering
\raisebox{-0.5\height}{\includegraphics[width=0.45\textwidth]{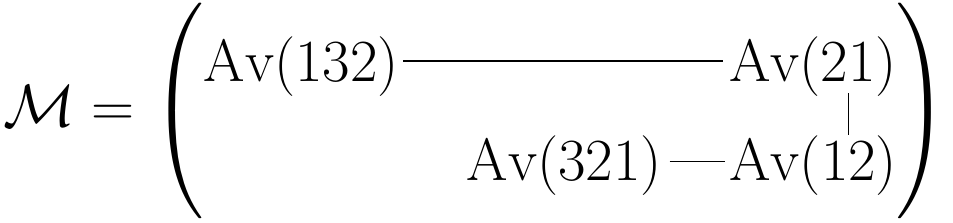}}
\hspace{0.5in}
\raisebox{-0.5\height}{\includegraphics[scale=0.45]{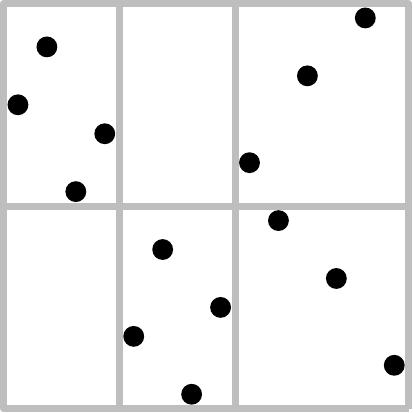}}
\hspace{0.5in}
\caption{A gridding matrix $\cM$ on the left and an $\cM$-gridded permutation on the right. Empty 
entries of $\cM$ are omitted and the edges of $G_\cM$ are displayed inside $\cM$.}
\label{fig:grid-class}
\end{figure}

\subparagraph*{Grid-width.} An \emph{interval family $\cI$} is a set of pairwise disjoint integer 
intervals. The 
\emph{intervalicity} of a set $A \subseteq [n]$, denoted by $\intv(A)$, is the size of the smallest 
interval family whose 
union is equal to~$A$. For a point set $S$ in the plane, let $\Pi_x(S)$ denote its projection on 
the $x$-axis and similarly $\Pi_y(S)$ its projection on the $y$-axis. We write $\intx(S)$ and 
$\inty(S)$ as short for $\intv(\Pi_x(S))$ and $\intv(\Pi_y(S))$, respectively. For a subset $S$ of 
the 
permutation diagram, the \emph{grid-complexity} of $S$ is the maximum of $\intx(S)$ and 
$\inty(S)$.

A \emph{grid tree} of a permutation~$\pi$ of length $n$ is a rooted binary tree $T$ with $n$ 
leaves, each leaf being labeled by a distinct point of the permutation diagram.  Let $\pi_v^T$ 
denote the point set of the labels on the leaves in the subtree of $T$ rooted in~$v$. The 
\emph{grid-width} of a vertex $v$ in $T$ is the grid-complexity of $\pi_v^T$, and the 
\emph{grid-width} of $T$, denoted by $\gw^T(\pi)$, is the maximum grid-width of a vertex of $T$. 
Finally, the \emph{grid-width} of a permutation $\pi$, denoted by $\gw(\pi)$, is the minimum of 
$\gw^T(\pi)$ over all grid trees~$T$ of~$\pi$.

We also consider a linear version of this parameter. We say that a rooted binary tree $T$ is a 
\emph{caterpillar} if each vertex is either a leaf or has at least one leaf as a 
child. The \emph{path-width} of a permutation $\pi$, denoted by $\pw(\pi)$, is the minimum 
of $\gw^T(\pi)$ over all caterpillar grid trees~$T$ of ~$\pi$.

We now provide a useful alternative definition of path-width. For permutations $\pi$ and $\sigma$ 
of length $n$, the \emph{path-width of $\pi$ in $\sigma$-ordering}, denoted by $\pw^\sigma(\pi)$ is 
the maximum grid-complexity attained by a set $\left\lbrace(\sigma_1, \pi_{\sigma_1}), \ldots, 
(\sigma_i, \pi_{\sigma_i})\right\rbrace$ for some $i \in [n]$.

\begin{lemma}
\label{lem:pathwidth-ordering}
A permutation $\pi$ of length $n$ has path-width $p$ if and only if the minimum value of 
$\pw^\sigma(\pi)$ over all permutations $\sigma$ of length $n$ is exactly $p$.
\end{lemma}
\begin{proof}
Suppose that $\pw(\pi) = p$ as witnessed by a caterpillar grid tree~$T$. Observe that all leaves 
of $T$ except for the deepest pair lie in different depths. Define $\sigma$ to simply order the 
labels by the depths of their leaves, defining arbitrarily the order of the two deepest leaves. 
Then every set $\lbrace(\sigma_1, \pi_{\sigma_1}), \ldots, (\sigma_i, \pi_{\sigma_i})\rbrace$ 
corresponds exactly to the set $\pi_v^T$ for some vertex $v$ of $T$.
	
In order to prove the other direction, we define a sequence of caterpillar trees $T_1, \ldots, T_n$ 
in the following way. Let $T_1$ be a single vertex labeled by $(\sigma_1, \pi_{\sigma_1})$. For 
$i>1$, let $T_i$ be the binary rooted tree with left child a leaf labeled by 
$(\sigma_i, \pi_{\sigma_i})$ and right child the tree $T_{i-1}$. The tree $T_n$ is a caterpillar 
grid tree of $\pi$ and for every inner vertex $v$ the set $\pi_v^T$ is equal to $\lbrace(\sigma_1, 
\pi_{\sigma_1}), \ldots,(\sigma_i, \pi_{\sigma_i})\rbrace$ for some $i$. The claim follows.
\end{proof}

Ahal and Rabinovich~\cite{Ahal2000} designed an algorithm for \PPM\ that runs in time 
$n^{O(\tw(G_\pi))}$ where $G_\pi$ is a certain graph associated to the pattern $\pi$ and 
$\tw(G_\pi)$ denotes the treewidth of $G_\pi$. The following theorem follows by combining this 
algorithm with the result of Jelínek et al.~\cite{Jelinek2018} who showed that up to a constant, 
$\gw(\pi)$ is equivalent to $\tw(G_\pi)$.

\begin{theorem}[Ahal and Rabinovich~\cite{Ahal2000}, Jelínek et al.~\cite{Jelinek2018}]
\label{thm:bounded-gw-algo}
Let $\pi$ be a permutation of length $k$ and $\tau$ a permutation of length $n$. The problem
whether $\tau$ contains $\pi$ can be solved in time $n^{O(\gw(\pi))}$.
\end{theorem}

Importantly, Theorem~\ref{thm:bounded-gw-algo} implies that \PPPM{$\cC$} is decidable in 
polynomial time whenever the class $\cC$ has bounded grid-width. In fact, we obtain all the 
polynomial-time solvable cases of \PPPM{$\cC$} in this paper via showing that $\cC$ has bounded 
grid-width.

\section{Monotone grid classes}
\label{sec:monotone}

This section is dedicated to proving that the complexity 
of \PPPM{$\Grid(\cM)$} is for a monotone gridding matrix $\cM$ determined by whether $G_\cM$ 
contains a cycle.

\begin{theorem}
\label{thm:monotone-grid}
For a monotone gridding matrix $\cM$ one of the following holds:
\begin{itemize}
\item Either $G_\cM$ is a forest, $\Grid(\cM)$ has bounded path-width and \PPPM{$\Grid(\cM)$}
can be decided in polynomial time, or
\item $G_\cM$ contains a cycle, $\Grid(\cM)$ has unbounded grid-width and \PPPM{$\Grid(\cM)$}
is NP-complete.
\end{itemize}
\end{theorem}

A \emph{consistent orientation} of a $k\times \ell$ monotone gridding matrix $\cM$ is a pair of 
functions $(c,r)$ such that $c\colon [k] \to \{-1,1\}$, $r\colon [\ell] \to \{-1,1\}$ and for every 
$i \in 
[k]$, $j \in [\ell]$ the value $c(i)r(j)$ is positive if $\cM_{i,j} = \;\Inc$ and negative if 
$\cM_{i,j} = \;\Dec$. 

Intuitively speaking, the purpose of the consistent orientation is to 
`orient' each nonempty cell of a monotone gridding matrix, so that $\Inc$-cells are oriented 
towards the top-right or towards the bottom-left, while the $\Dec$-cells are oriented towards the 
top-left or the bottom-right. The orientation is consistent in the sense that in a given column $i$, 
either all the nonempty cells are oriented left-to-right (if $c(i)=1$) or right-to-left (if 
$c(i)=-1$), while in a row $j$, they are all oriented bottom-to-top (if $r(j)=1$) or top-to-bottom 
(if $r(j)=-1$). See Figure~\ref{fig:path-creation} (left) for an example of a gridding matrix with 
a consistent orientation.

Not every monotone gridding matrix has a consistent orientation: consider, e.g., a $2\times 2$ 
gridding matrix with three cells equal to $\Inc$ and one cell equal to~$\Dec$. However,
Albert et al.~\cite{Albert2013} observed that we can transform a given gridding matrix
$\cM$ into another, similar gridding matrix that has a consistent orientation.
%Before we are able to state their observation, we introduce an important operation on gridding 
%matrices. 
Let $\cM$ be a $k \times \ell$ monotone gridding matrix and $q$ a positive integer. The 
\emph{refinement $\cM^{\times q}$} of $\cM$ is the $qk \times q\ell$ matrix obtained from $\cM$ by 
replacing each $\Inc\,$-entry by a $q\times q$ diagonal matrix with all the non-empty entries equal 
to $\Inc$, each $\Dec\,$-entry by a $q\times q$ anti-diagonal matrix with all the non-empty entries 
equal to $\Dec$ and each empty entry by a $q \times q$ empty matrix. It is easy to see that 
$\Grid(\cM^{\times q})$ is a subclass of $\Grid(\cM)$. Moreover, if $G_\cM$ is a forest then 
$\Grid(\cM^{\times q}) = \Grid(\cM)$ and $G_{\cM^{\times q}}$ is a forest as well.

\begin{lemma}[Albert et al.~\protect{\cite[Proposition 4.1]{Albert2013}}]\label{lem-albert}
For every monotone gridding matrix $\cM$, the refinement $\cM^{\times2}$ admits a
consistent orientation.
\end{lemma}

We remark that Albert et al.~\cite{Albert2013} use a slightly different way of defining a permutation 
class from a given gridding matrix: specifically, their classes only contain permutations in which 
the entries represented by each cell of the gridding can be placed on a segment with slope $+1$ 
or~$-1$. However, as Lemma~\ref{lem-albert} is a claim about gridding matrices and not about 
permutation classes, we can use it here for our purposes.

Now we provide bounds on the width parameters of monotone grid classes depending on the 
structure of their cell graphs.

\begin{proposition}
\label{prop:small-pw}
Let $\cM$ be a $k \times \ell$ monotone gridding matrix that has a consistent orientation. If 
$G_\cM$ 
is a forest then every permutation of $\Grid(\cM)$ has path-width at most $\max(k,\ell)$.
\end{proposition}
\begin{proof}
Let $\pi$ be a permutation from $\Grid(\cM)$ of length $n$ together with an $\cM$-gridding $I^0_1 < 
\cdots < I^0_k$ and  $J^0_1 < \cdots < J^0_\ell$. We fix a consistent orientation $(c,r)$ for the 
matrix $\cM$.

For $I \subseteq I^0_i$, the \emph{extremal point} of $I$ is the rightmost point of $\pi$ 
restricted to $I\times [n]$ if $c(i)$ is positive, the leftmost point otherwise. Similarly for $J 
\subseteq J^0_j$, the extremal point of $J$ is the topmost point of $\pi$ restricted to 
$[n]\times J$ if $r(j)$ is positive, the bottommost point otherwise. Importantly, the definition of 
consistent orientation guarantees that if $I \times J$ contains the extremal points of both 
$I$ and $J$, then these two points must actually be the same point.

We now construct an ordering $\sigma$ of length $n$ and a sequence of interval families 
 $(\cI^m, \cJ^m)$ for $m \in [n]$ such that for every $m$
\begin{itemize}
\item $\cI^m$ contains $k$ (possibly empty) intervals $I_1^m < \cdots < I_k^m$ and $\cJ^m$ contains 
$\ell$ (possibly empty) intervals $J_1^m < \cdots < J_\ell^m$,
\item $I^m_s \subseteq I^0_s$ and $J^m_t \subseteq J^0_t$ for every $s$ and $t$, and
\item for $P = \{(\sigma_{m+1},
\pi_{\sigma_{m+1}}), (\sigma_{m+2}, \pi_{\sigma_{m+2}}),\allowbreak\dots,\allowbreak(\sigma_n, 
\pi_{\sigma_n})\}$ we have $\Pi_x(P) = \bigcup\cI^m$ and $\Pi_y(P) = \bigcup\cJ^m$.
\end{itemize}
The third condition then implies that $\pw^{\sigma^R}(\pi) \leq \max(k,\ell)$ which proves 
the proposition.

\newcommand{\ogm}{${\vec G}_\cM$}
Suppose that we have already defined the sequences up to $m$. 
Our goal is to find indices $i\in[k]$ and $j\in[\ell]$ such that $I^m_i\times J_j^m$ contains the 
extremal point of both $I^m_i$ and~$J^m_j$. To this end, we transform $G_\cM$ into a directed 
graph \ogm\ as follows: suppose that a column $i$ of $\cM$ contains nonempty cells 
$c_1,c_2,\dotsc,c_p$ ordered bottom to top. These cells form a path in~$G_\cM$. Assume that the 
extremal point of $I_i^m$ is in the cell~$c_j$. We then orient the edges of the path $c_1c_2\dotsc 
c_p$ so that they all point towards~$c_j$, that is, for every $a<j$ the edge $\{c_a,c_{a+1}\}$ is 
oriented upwards, while for $a\ge j$ it is oriented downwards. If $I_i^m$ is empty, we 
remove the edges of the path $c_1\dotsb c_p$ from \ogm\ entirely. We perform an analogous operation 
for every row $j$ of $\cM$, orienting the edges of the corresponding path in $G_\cM$ so that they 
point towards the cell containing the extremal point of~$J_j^m$, or removing the edges entirely if 
$J_j^m$ is empty.

% Let $G^m$ be an auxiliary oriented 
% bipartite graph whose vertices are $s_1, \ldots, s_k, t_1,\ldots, t_\ell$, and there is an 
% edge $(s_i, t_j)$ whenever the extremal point of $I^m_i$ lies in the $j$-th row of the 
% $\cM$-gridding and an edge $(t_j, s_i)$ whenever the extremal point of $J^m_j$ lies in the $i$-th 
% column of the $\cM$-gridding.
% 
% Note that every vertex corresponding to a non-empty row has exactly one outgoing edge to a 
% non-empty column and vice versa. Therefore, we can find a cycle in $G^m$ by starting in arbitrary 
% non-empty column and following the outgoing edges. However, observe that there cannot be a cycle of 
% length larger than 2 as such cycle would imply a cycle in the graph $G_\cM$. 

As \ogm\ is a forest, each of its components has a sink. Ignoring the components of \ogm\ that 
correspond to intersections of empty rows with empty columns, and choosing a sink of one of the 
remaining components, we find a column $i$ and a row $j$ of $\cM$ such that both $I^m_i$ and $J^m_j$ 
have their extremal points in $I^m_i\times J^m_j$. As we observed, this must be the same point~$p$. 

We set $\sigma_{m+1}$ to be $p.x$ and we define the interval families $\cI^{m+1}$, $\cJ^{m+1}$ by 
removing $p.x$ from $\cI^{m}$ and $p.y$ from~$\cJ^{m}$. Observe that $\cI^{m+1}$, 
$\cJ^{m+1}$ are well-defined because $p$ was the extremal point of both intervals $I^m_i$ and 
$J^m_j$.
\end{proof}

\begin{lemma}
\label{lem:long-path-width} 
Let $\cM$ be a monotone gridding matrix such that $G_\cM$ contains a path of length $k$. Then there 
exists a permutation $\pi \in \Grid(\cM)$ with grid-width at least $\frac{k}{4}$.
\end{lemma}
\begin{proof}
Without loss of generality, we assume that $G_\cM$ is a path of length $k$, i.e. all the other 
entries are empty. We construct an $\cM$-gridded permutation $\pi$ of length $n=k^3$ such that 
$\gw(\pi) \geq \frac{k}{4}$ in the following way. The permutation $\pi$ is a union of point sets 
$B_1, \dots, B_k$, called blocks, such that each $B_i$ has size $k^2$ and is contained in the cell 
of the $\cM$-gridding corresponding to the $i$-th vertex of the path. Moreover, for every $i$, the 
point set $B_i \cup B_{i+1}$ forms a horizontal or vertical alternation, depending on whether their 
respective cells share the same row or column. If the relation between point sets $B_i$ and $B_j$ 
for $|i-j| \ge 2$ is not fully determined by the position of their respective cells (they share the 
same row or column) they can be interleaved in arbitrary way as long as $B_i \cup B_{i+1}$ forms an 
alternation for every~$i$.

Let~$T$ be the grid-tree of $\pi$ with minimum grid-width. 
By standard arguments, there is a vertex $v \in T$ such that the subtree of $v$ contains at least 
$\frac{n}{3}$ and at most $\frac{2n}{3}$ leaves. Let $S$ be the subset of the permutation diagram 
of $\pi$ defined by these leaves.

Let the \emph{density} of a block~$B_i$ be the ratio $\frac{|S\cap B_i|}{|B_i|}$, 
denoted by~$d_i$. We claim that the densities of consecutive blocks cannot differ too much, in 
particular that the difference $|d_i - d_{i+1}|$ is at most $\frac{1}{4k}$. Without loss of 
generality, assume that $B_i$ and $B_{i+1}$ share the same column and that $d_{i+1} > d_i$. If the 
density of $B_i$ and $B_{i+1}$ differed by at least $\frac{1}{4k}$, there would be at least $k/4$ 
more points of $S$ in $B_{i+1}$. Since $B_i \cup B_{i+1}$ forms a vertical alternation, we can 
pair each 
point of $B_i$ with the nearest point to the right of it in $B_{i+1}$. Then at least $k/4$ of these 
pairs would consist of a point of $B_{i+1}$ in $S$ and a point of $B_i$ outside of~$S$. The 
intervalicity of $\Pi_x(S \cap (B_i \cup B_{i+1}))$ would thus be at least $k/4$ and that 
is a contradiction.

We now show that each block of $\pi$ contains both a point in $S$ and a point outside of $S$.
Suppose that the block $B_i$ is fully contained in $S$, i.e. the density $d_i$ is equal to~$1$. 
Since the differences in densities of consecutive blocks cannot be larger than $\frac{1}{4k}$, 
every block has density at least $1 - k\frac{1}{4k} = \frac{3}{4}$. But then $S$ must contain at 
least $\frac{3}{4}n$ points, which is a contradiction, since $\frac{1}{3}n\le |S|\le\frac23n$. 
Similarly, there cannot be any block whose density is equal to 0, since then every block would have 
density at most $\frac14$, again a contradiction.

Therefore, every block $B_i$ contains both a point from $S$ and a point outside of~$S$. The number 
of rows and columns of $\cM$ spanned by the path is at least $k$ since every step on the path 
introduces either a new row or a new column. Let us therefore, without loss of generality, assume 
that the path spans at 
least $\frac{k}{2}$ columns. Since each of the columns contains point both from $S$ and outside of 
$S$, any consecutive interval of $\Pi_x(S)$ cannot intersect more than $2$ of these 
columns and the grid-complexity of $S$ is at least~$\frac{k}{4}$. This means that the vertex $v$ of 
$T$ has grid-width at least $\frac k4$, therefore the grid-width of $T$ is at most $\frac k4$, and 
since $T$ was the optimal grid-tree for $\pi$, the grid-width of $\pi$ is  at least $\frac k4$, as 
claimed.
\end{proof}

We proceed to show that whenever $G_\cM$ 
contains a cycle, $\Grid(\cM)$ contains grid subclasses with arbitrarily long paths in their cell 
graph. In fact, all the properties we show about cyclic grid classes rely structurally only 
on the existence of these long paths.

\begin{figure}
\centering
\raisebox{-0.5\height}{\includegraphics[scale=0.6]{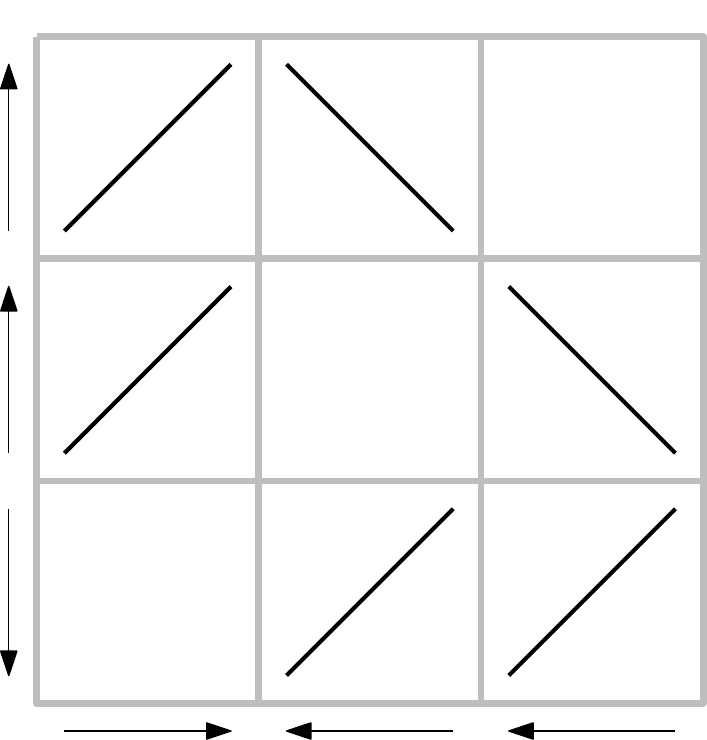}}
\hspace{0.05in}
\raisebox{-0.5\height}{\includegraphics[scale=0.6]{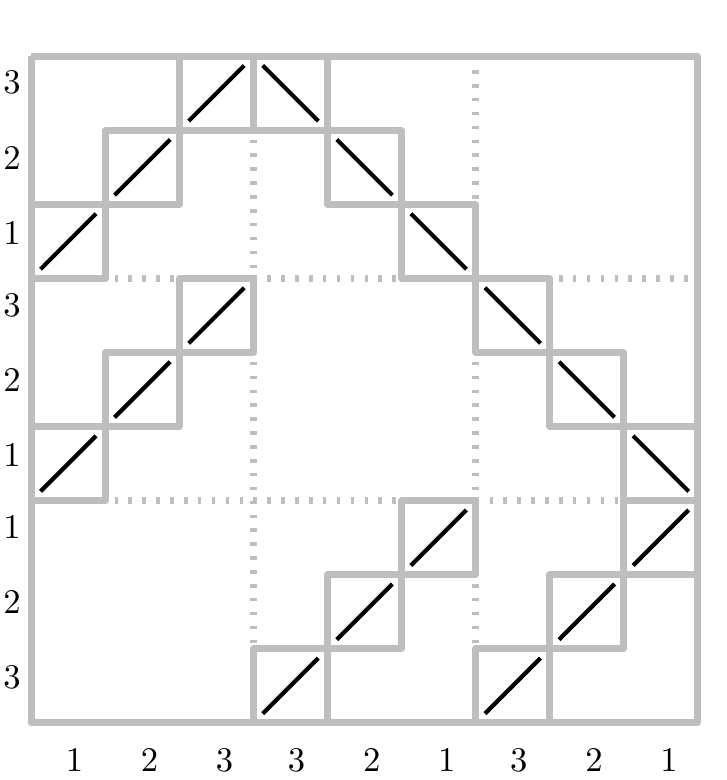}}
\hspace{0.05in}
\raisebox{-0.5\height}{\includegraphics[scale=0.6]{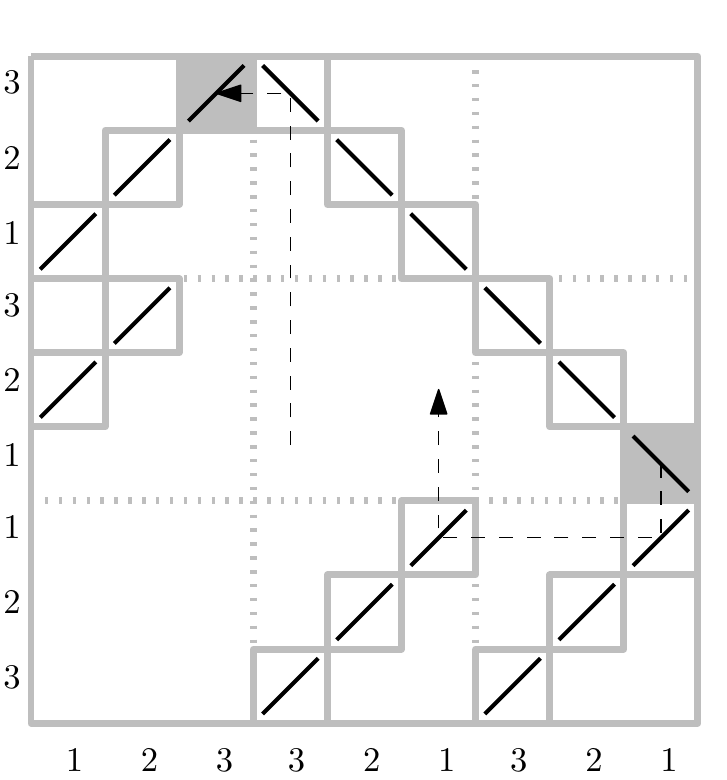}}
\caption{The matrix modification process in the proof of Lemma~\ref{lem:long-path}. A gridding 
matrix $\cM$ with a consistent orientation on the left, its refinement 
$\cM^{\times 3}$ with labeled rows and columns in the middle and the gridding matrix $\cM_3$ on the 
right. The endpoints of path in $G_{\cM_3}$ are highlighted.} \label{fig:path-creation}
\end{figure}

\begin{lemma}
\label{lem:long-path}
Let $\cM$ be a monotone gridding matrix such that $G_\cM$ contains a cycle. For every $p \geq 1$, 
there is a gridding matrix $\cM_p$ such that $\Grid(\cM)$ contains $\Grid(\cM_p)$ as a subclass 
and moreover, $G_{\cM_p}$ is a proper turning path of length at least~$p$. Furthermore, given $\cM$ 
and the integer $p$ we can compute $\cM_p$ in polynomial time.
\end{lemma}
\begin{proof}
Let $C$ be a cycle in $G_\cM$. We can without loss of generality assume that $C$ is proper turning 
and that every cell outside of $C$ is empty, otherwise we could replace all the cells of $\cM$ that 
do not correspond to the turns of $C$ by empty cells.

The proof is illustrated in Figure~\ref{fig:path-creation}. Fix a consistent 
orientation $(c,r)$ of $\cM$ and recall the definition of the refinement 
$\cM^{\times p}$. We proceed by labeling the rows and columns of $\cM^{\times p}$ using the set 
$[p]$. The $p$-tuple of columns created from the $i$-th column of $\cM$ is labeled 
$1,2,\ldots, p$ from left to right if $c(i)$ is positive, and right to left otherwise. Similarly, 
the $p$-tuple of rows created from the $j$-th row of $\cM'$ is labeled $1,2,\ldots, p$ 
from bottom to top if $r(j)$ is positive, and top to bottom otherwise. The \emph{characteristic of 
a cell} in $\cM^{\times p}$ is the pair of labels given to its column and row.
The consistent orientation guarantees that each non-empty cell in $\cM^{\times p}$ 
has a characteristic of form $(s,s)$ for some $s \in [p]$. Therefore, $\cM^{\times p}$ consists 
exactly of $p$ components, each being a copy of $\cM$.

The \emph{$(i,j)$-block} of $\cM^{\times p}$ is the $p \times p$ submatrix corresponding to the 
$(i,j)$-cell of $\cM$. We pick an arbitrary non-empty cell $(i,j)$ of $\cM$ and obtain a matrix 
$\cM_p$ by replacing $(i,j)$-block in $\cM^{\times p}$ with the matrix whose only non-empty 
entries have the characteristic $(s,s+1)$ for all $s \in [p-1]$ and are of the same type as 
$\cM_{i,j}$. $\Grid(\cM_p)$ is a subclass of $\Grid(\cM)$ since the modified $(i,j)$-block 
corresponds to shifting the original (anti-)diagonal matrix by one row either up or down, 
depending on the orientation of the $j$-th row of $\cM$. 

%We claim that we connected all the $p$ copies of $\cM'$ into a single long path. Let $(i',j)$ and 
%$(i,j')$ be the two neighbors of $(i,j)$ in $C'$. In $G_{\cM_p}$, the cell with characteristic 
%$(s,s+1)$ in the $(i,j)$-block connects the cell with characteristic $(s,s)$ in the $(i,j')$-block 
%to the cell with characteristic $(s+1, s+1)$ in the $(i',j)$-block. Thus, if we start in the cell 
%of characteristic $(1,1)$ in the $(i',j)$-block we eventually walk through all the non-empty cells 
%of $\cM_p$.
Observe that we connected all the $p$ copies of $\cM$ into a single long path. Moreover, we in fact 
described an algorithm how to compute $\cM_p$ in polynomial time.
\end{proof}

Combining Lemmas~\ref{lem:long-path-width} and~\ref{lem:long-path}, we directly obtain the 
following corollary.

\begin{corollary}\label{cor:unbounded-gw}
Let $\cM$ be a monotone gridding matrix such that $G_\cM$ contains a cycle. Then $\Grid(\cM)$ has 
unbounded grid-width.\qed
\end{corollary}

%Let us define a special type of gridding matrices whose cell graph is a proper turning path, the 
%so-called staircase classes. 
In order to state and prove our hardness result that contrasts Proposition~\ref{prop:small-pw}, we 
need to introduce several definitions.
Let $\cC$ and $\cD$ be permutation classes and $k$ a positive integer.
The \emph{$(\cC, \cD)$-staircase of $k$ steps} is the $(k + 1)\times k$ gridding matrix $\St^k(\cC, 
\cD)$ such that the $(i,i)$-cell contains $\cC$, the $(i+1,i)$-cell contains $\cD$ for every $i 
\in [k]$, and every other cell is empty. The staircase classes have been studied by Albert et 
al.~\cite{Albert2019} in the context 
of determining the growth rates of certain permutation classes.

%Observe that $\Grid(\St^k(\Inc, \Inc))$ is a subclass of $\Av(321)$ for every $k$. Moreover, it is 
%not hard to show that in fact $\Av(321) = \bigcup_{k=1}^\infty \Grid(\St^k(\Inc,\Inc)))$. 

Let $\cM, \cN$ be $k \times \ell$ gridding matrices, let $\pi$ be an $\cM$-gridded permutation 
and let $\tau$ an $\cN$-gridded permutation. We say that $\tau$ contains a \emph{grid-preserving 
copy} of $\pi$, if there is an occurrence of $\pi$ in $\tau$ such that the elements from the 
$(i,j)$-cell of the $\cM$-gridding of $\pi$ are mapped to elements in the $(i,j)$-cell of the 
$\cN$-gridding of $\tau$ for every $i$ and $j$.

To complete the proof of Theorem~\ref{thm:monotone-grid}, we need to show that \PPPM{$\Grid(\cM)$} 
is NP-complete whenever $G_\cM$ contains a cycle. Our argument is based upon a construction of 
Jelínek and Kynčl~\cite{Jelinek2017}, who proved that that \PPPM{$\Av(321)$} is NP-complete. The 
following theorem describes one of the key steps of their proof.

\begin{theorem}[Jelínek and Kynčl~\cite{Jelinek2017}]
\label{thm:jelinek-hard}
Let $\Phi$ be a 3-CNF formula with $v$ variables and $c$ clauses. There is a polynomial time 
algorithm that outputs a $\St^{2c+1}(\Inc, \Inc)$-gridded permutation $\pi$ and a 
$\St^{2c+1}(\Av(321), \Inc)$-gridded permutation $\tau$ 
such that $\Phi$ is satisfiable  if and only if there is a grid-preserving copy of $\pi$ in $\tau$. 
Additionally, the longest increasing subsequences of the $(1,1)$-cell of $\pi$ and the 
$(1,1)$-cell of $\tau$ are both of length $2v$.
\end{theorem}

Let us modify $\pi$ and $\tau$ such that any embedding that maps the $(1,1)$- 
cell of $\pi$ to the $(1,1)$-cell of $\tau$ must already be grid-preserving.

The \emph{lane of $k$ steps} is the $\St^{k}(\Inc, \Inc)$-gridded permutation such that each 
non-empty cell of the staircase contains exactly 2 points and two neighboring cells in the same row 
form a copy of 1423 while two neighboring cells in the same column form a copy of 1342. The 
intuition here is that as the lane intersects two adjacent cells in the same row, the two elements 
in the left cell (corresponding to $1$ and $4$ of the pattern $1423$) are `sandwiching' the two 
elements of the right cell from above and from below. Similarly, for two cells in the same column, 
the bottom part of the lane sandwiches the top part from the left and from the right.

For a $\St^{k}(\cC, \cD)$-gridded permutation $\pi$, the result of \emph{confining $\pi$} is the 
following gridded permutation $\pi'$. Let $\cN$ be the gridding matrix obtained by replacing every 
non-empty entry of $\St^k(\cC, \cD)$ with a $3\times3$ diagonal matrix whose (1,1)-cell and 
(3,3)-cell contain $\Inc$ and the $(2,2)$-cell is equal to the original entry. Observe that 
$\cN$ consists of 3 components, namely a copy of $\St^{k}(\cC, \cD)$ `sandwiched' between two 
copies of $\St^{k}(\Inc, \Inc)$. We obtain $\pi'$ by placing $\pi$ in the middle copy and two lanes 
of $k$ steps in the outer copies. Finally, we unify each $3 \times 3$ block of the $\cN$-gridded 
$\pi'$ into a single cell. See the left part of Figure~\ref{fig:np-reduct}.

Let $\pi$ and $\tau$ be the gridded permutations from Theorem~\ref{thm:jelinek-hard} and $\pi', 
\tau'$ their confined versions. Suppose there is an occurrence of $\pi'$ in $\tau'$ that maps the 
$(1,1)$-cell of $\pi'$ to the $(1,1)$-cell of $\tau'$. It must map the first two points of the two 
lanes in $\pi'$ to the first two points of the lanes in $\tau'$ since the longest increasing 
subsequences in the $(1,1)$-cells of $\pi$ and $\tau$ are of the same length. But then
the whole lanes from $\pi'$ map to the respective lanes in $\tau'$; this is because the pair of 
elements in a given cell of the lane in $\pi'$ (or $\tau'$) sandwiches exactly two other elements of 
$\pi'$ (or $\tau'$), namely the two elements of the same line belonging to the following cell.
This then forces the elements of each non-empty $(i,j)$-cell of $\pi'$ to map to the 
$(i,j)$-cell of~$\tau'$.

\begin{proposition}
\label{prop:cyclic-grid}
Let $\cM$ be a monotone gridding matrix such that $G_\cM$ contains a cycle. Then 
%$\Grid(\cM)$ has unbounded grid-width and 
\PPPM{$\Grid(\cM)$} is NP-complete.
\end{proposition}
\begin{proof}
% We shall prove the first part by contradiction. Suppose that there is an integer $g$ 
% such that every permutation $\pi \in \Grid(\cM)$ has grid-width at most $g$. By 
% Lemma~\ref{lem:long-path}, there is a monotone gridding matrix $\cM'$ such that $\Grid(\cM')$ is a 
% subclass of $\Grid(\cM)$ and $G_{\cM'}$ contains a path of length $4g + 1$. Therefore, there is a 
% permutation $\pi\in \Grid(\cM')$ with grid-width at least $\frac{4g+1}{4} > g$ due to 
% Lemma~\ref{lem:long-path-width}.

We describe the reduction from 3-SAT to \PPPM{$\Grid(\cM)$}. Let $\Phi$ be a given 3-CNF 
formula with $v$ variables and $c$ clauses. Using Theorem~\ref{thm:jelinek-hard}, we compute 
gridded permutations $\pi$ and $\tau$ such that $\Phi$ is satisfiable if and only if there is a 
grid-preserving copy of $\pi$ in $\tau$. As we have shown, we can assume that 
any occurrence of $\pi$ in $\tau$ that maps the $(1,1)$-cell of $\pi$ to the $(1,1)$-cell of $\tau$ 
must be grid-preserving. Furthermore, we obtain a monotone gridding matrix $\cM'$ such that 
$\Grid(\cM')$ is a subclass of $\Grid(\cM)$ and $G_{\cM'}$ is a proper turning path $P$ of length 
$4c+2$ by application of Lemma~\ref{lem:long-path}. Without loss of generality, we may assume that 
the first two cells of this path occupy the same row.

\begin{figure}
\centering
\raisebox{-0.5\height}{\includegraphics[width=0.4\textwidth]{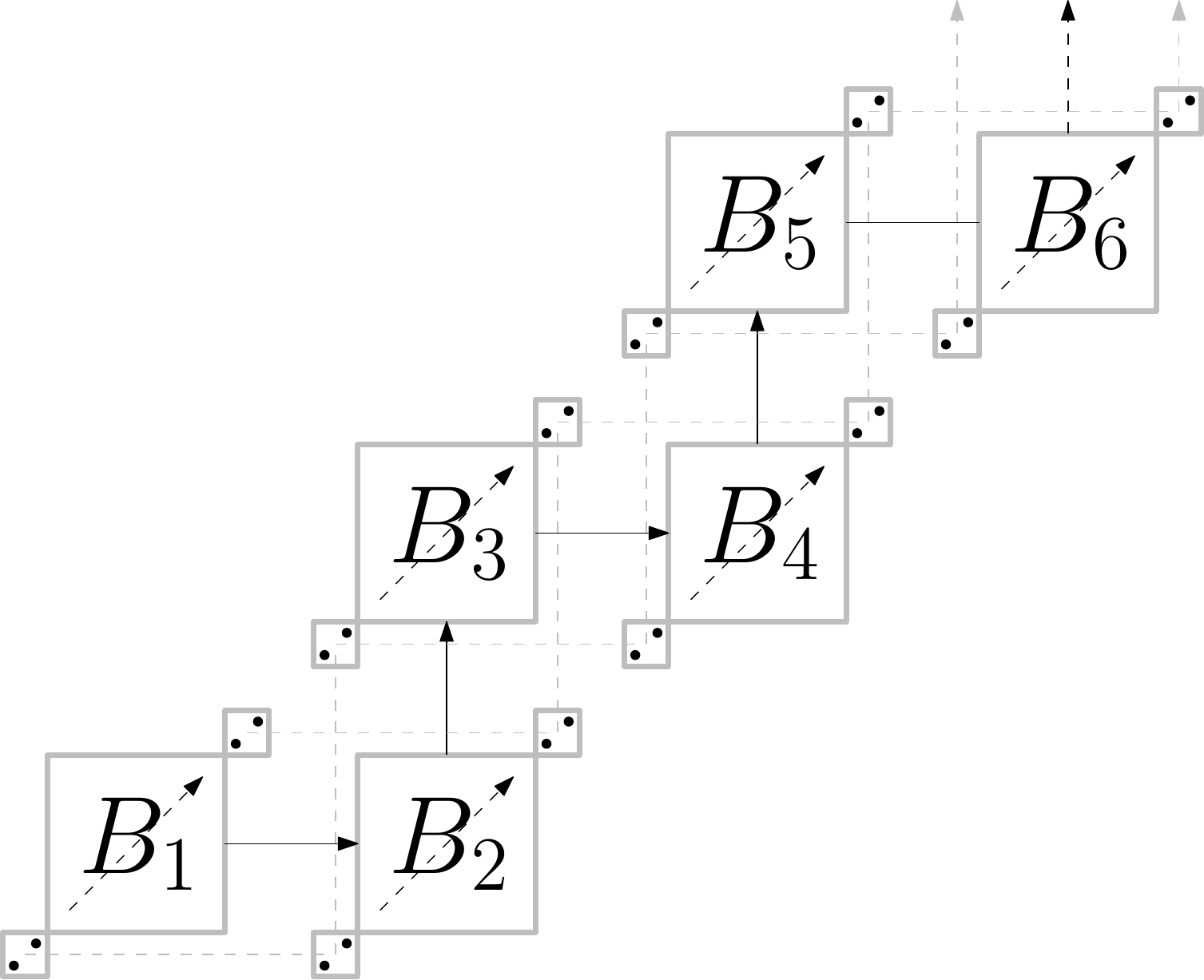}}
\hspace{0.5in}
\raisebox{-0.5\height}{\includegraphics[width=0.4\textwidth]{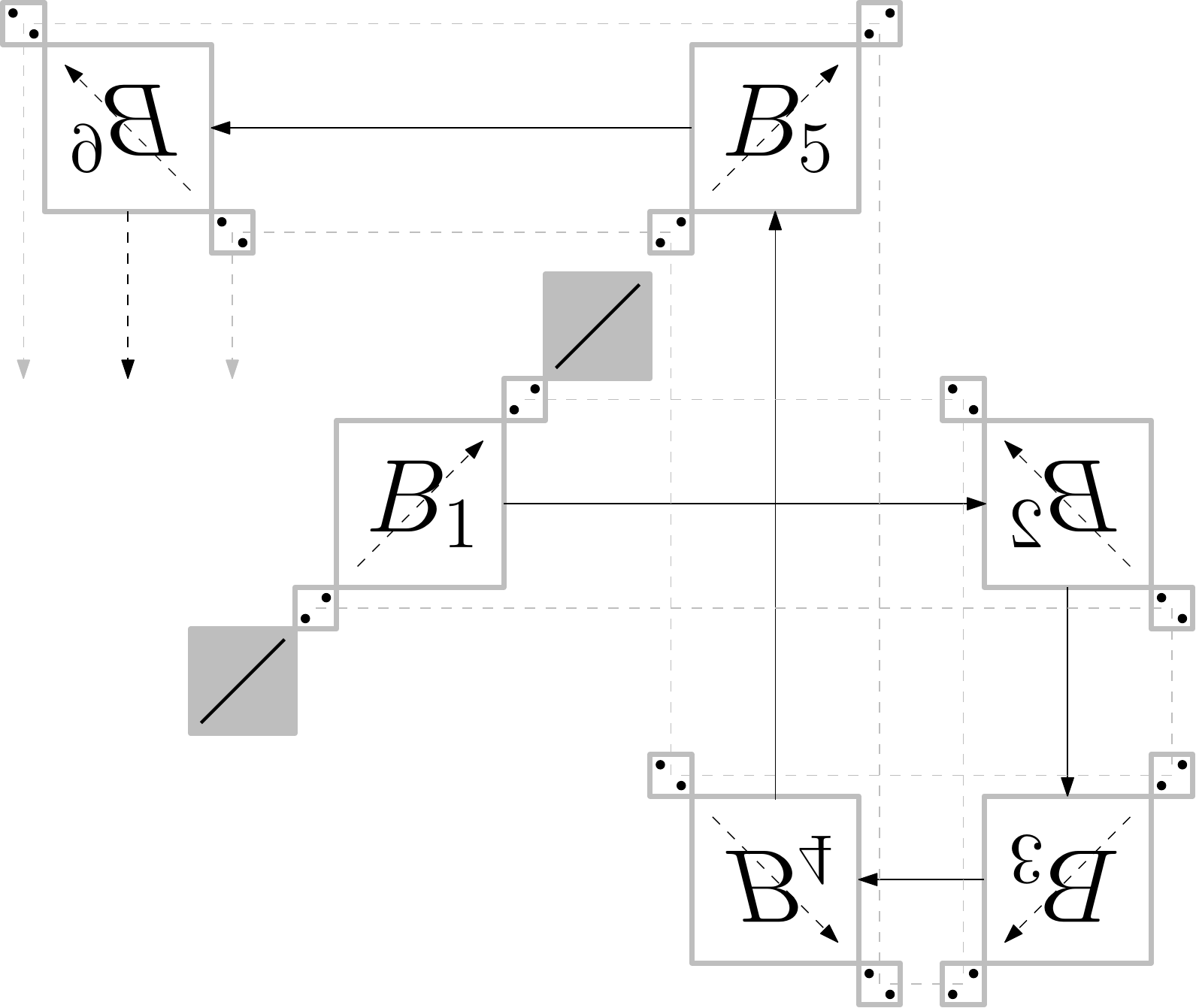}}
\caption{Situation in the proof of Proposition~\ref{prop:cyclic-grid}. The confined permutation 
$\pi$ on the left, the permutation $\pi^\star$ obtained by adding anchors (highlighted) to the 
$(f,g)$-transform 
$\pi'$ on the right.}
\label{fig:np-reduct}
\end{figure}

First, we aim to construct gridded permutations $\pi'$ and $\tau'$ such that $\pi' \in \Grid(\cM')$ 
and there is a grid-preserving copy of $\pi'$ in $\tau'$ if and only if there is a grid-preserving 
copy of $\pi$ in~$\tau$. We essentially aim to generalize the `twirl' operation used by Jelínek and 
Kynčl~\cite{Jelinek2017}. A \emph{signed permutation of length $k$} is a permutation of length 
$k$ in which each entry is additionally equipped with a sign.
Let $\cQ$ be an arbitrary $k \times \ell$ gridding matrix and let $f$ be a signed permutation of 
length $k$ and $g$ be a signed permutation of length $\ell$. The \emph{$(f,g)$-transform} of $\cQ$ 
is a gridding matrix $\cQ'$ such that $\cQ'_{i,j} = (\cQ_{|f(i)|, |g(j)|})^o$ where the operation 
$o$ is identity if both $f(i)$ and $g(j)$ are positive, reversal if only $f(i)$ is negative, 
complement if only $g(j)$ is negative and reverse complement if both $f(i)$ and $g(j)$ are negative.
In other words, $\cQ'$ is created by permuting the rows and columns of $\cQ$ according to the 
permutations $f$ and $g$ while also flipping around those with negative signs. 

Analogously, we can define a transformation of a $\cQ$-gridded permutation into a $\cQ'$-gridded 
one. The \emph{$(f,g)$-transform} of a $\cQ$-gridded permutation $\sigma$ is the gridded 
permutation $\sigma'$ obtained by permuting the columns of the gridding according to $f$, rows 
according to $g$, replacing the point set of the $i$-th column with its reversal if $f(i)$ is 
negative and then replacing the point set of the $j$-th row with its complement if $g(j)$ is 
negative. 
%Notice that $\sigma'$ is indeed a $\cQ'$-gridded permutation.

For any fixed $f$ and $g$, there is indeed a grid-preserving occurrence of $\pi$ in $\tau$ 
if and only if there is a grid-preserving occurrence of the $(f,g)$-transform of $\pi$ in the 
$(f,g)$-transform of~$\tau$. Let us now define $f$ and $g$ such that the $(f,g)$-transform of 
$\St^{2c+1}(\Inc, \Inc)$ is exactly the gridding matrix $\cM'$. Let $(c,r)$ be a consistent 
orientation of $\cM'$, let $s_1, \dots, s_{2c+2}$ be the sequence of column indices in the order 
as visited by the path $P$ and let $t_1, \dots, t_{2c+1}$ be the sequence of row indices in the 
order as 
visited by $P$. We set $f(s_i) = c(i)i$ and $g(t_j)=r(j)j$. Observe that the $i$-th cell of 
the path in the staircase is mapped by the $(f,g)$-transform precisely to the $i$-th cell of $P$. 
Moreover, the consistent orientation guarantees that the type of the entry obtained via 
$(f,g)$-transform agrees with the type of the entry in $\cM'$. Let $\cN$ be the $(f,g)$-transform of
$\St^{2c+1}(\Av(321), \Inc)$.

We set $\pi'$ to be the $(f,g)$-transform of $\pi$ and $\tau'$ to be the $(f,g)$-transform of 
$\tau$. As we already argued, $\pi'$ is an $\cM'$-gridded permutation, $\tau'$ is an $\cN$-gridded 
permutation and there is a grid-preserving occurrence of $\pi'$ in $\tau'$ if and only if there is 
a grid-preserving occurrence of $\pi$ in $\tau$. Moreover it is still true that if there is an 
occurrence of $\pi'$ in $\tau'$ that maps the $(s_1,t_1)$-cell to the $(s_1,t_1)$-cell of $\tau'$ 
then it must be grid-preserving. By the very same argument as before, in any map of the 
$(s_1,t_1)$-cell of $\pi'$ the beginning of the $(f,g)$-transformed lanes must map to the 
beginnings of the $(f,g)$-transformed lanes in $\tau'$. As before, this forces the whole mapping to 
be grid-preserving.

To summarize the argument so far, we used Theorem~\ref{thm:jelinek-hard} to reduce the NP-complete 
3-SAT problem to the problem of deciding whether a $\St^{2c+1}(\Inc, \Inc)$-gridded 
permutation $\pi$ has a grid-preserving occurrence in a $\St^{2c+1}(\Av(321), \Inc)$-gridded 
permutation $\tau$. Then, by means of the $(f,g)$-transform, we showed that this is equivalent to 
finding a grid-preserving occurrence of an $\cM'$-gridded permutation $\pi'$ in an $\cN$-gridded 
permutation~$\tau'$.

As the final step of our argument, we transform $\pi'$ and $\tau'$ into permutations $\pi^\star 
\in \Grid(\cM')$ and $\tau^\star \in \Grid(\cN)$ (that are no longer gridded) such that $\tau^\star$ 
contains $\pi^\star$ if and only if there is a grid-preserving copy of $\pi'$ in~$\tau'$. This will 
imply that \PPPM{$\Grid(\cM')$} is NP-complete, and therefore \PPPM{$\Grid(\cM)$} is NP-complete as 
well.

Let $p$ be the length of the longest monotone subpermutation of $\tau'$. Suppose that the 
$(s_1,t_1)$-cell of $\cM'$ contains $\Inc$, the other case being symmetric. We obtain $\pi^\star$ by 
inserting two increasing sequences of length $p+1$ into $\pi'$, one directly to the left and below 
the $(s_1,t_1)$-cell, called a lower anchor, and one directly above and to the right of the 
$(s_1,t_1)$-cell, called an upper anchor. We perform the same modification on $\tau'$ to obtain 
$\tau^\star$. See Figure~\ref{fig:np-reduct}.

Clearly, if $\tau'$ contains a grid-preserving copy of $\pi'$, then $\tau^\star$ contains 
$\pi^\star$. Let us prove the converse.
Fix any embedding $E$ of $\pi^\star$ into $\tau^\star$. Since there is no increasing subsequence of 
length $p+1$ in $\tau'$, at least $p+2$ of the points from the anchors in $\pi^\star$ must map to 
the anchors of $\tau^\star$. In particular, there is a point of the upper anchor in $\pi^\star$ 
mapped to the upper anchor in $\tau^\star$ and the same holds for the lower anchors. This implies 
that $E$ maps the whole copy of the initial block of $\pi'$ inside $\pi^\star$ to the initial 
block of~$\tau^\star$. We claim that $E$ in fact maps the whole initial block of $\pi'$ inside 
$\pi^\star$ to the initial block of $\tau'$ inside~$\tau^\star$, except perhaps the rightmost point 
and the leftmost point of the initial block of $\pi'$ (which belong to the lanes). 
Suppose for contradiction that the two leftmost points $q$ and $r$ of the initial block of $\pi'$, 
i.e. the beginning of one of the lanes, map both to the lower anchor. The second block of the 
lane must lie in the vertical interval between $q$ and $r$ either to the right of both of them or 
to the left of both of them. But for any two points of the lower anchor in $\tau^\star$ there are 
no such points to map $q$ and $r$ to. The same is true for the beginning of the second lane. 

From this argument, it actually follows that we can modify the embedding $E$ so that it maps the 
leftmost and rightmost point of the initial block of $\pi'$ to the leftmost and rightmost point of 
the initial block of~$\tau'$. By the previous arguments, the restriction of $E$ to $\pi'$ yields a 
grid-preserving embedding of $\pi'$ into $\tau'$, which concludes the proof.
\end{proof}

We have now completed the proof of Theorem~\ref{thm:monotone-grid}: the first part of the theorem 
follows by combining Proposition~\ref{prop:small-pw} with Theorem~\ref{thm:bounded-gw-algo}, and 
the second part follows from Corollary~\ref{cor:unbounded-gw} and 
Proposition~\ref{prop:cyclic-grid}.

%This concludes the proof of Theorem~\ref{thm:monotone-grid}.
% We remark that there is an interesting 
%connection to structural properties of permutation classes as Murphy and Vatter~\cite{Murphy2002} 
%proved that a monotone grid class $\Grid(\cM)$ is partially well-ordered if and only if $G_\cM$ is 
%a forest.

\section{General grid classes}
\label{sec:general}

In this section, we generalize the results of Theorem~\ref{thm:monotone-grid} to any gridding 
matrix whose every entry has bounded grid-width. Note that a `bumper-ended' path is a certain kind 
of path in $G_\cM$ whose definition we provide later.

\begin{theorem}
\label{thm:general-grid}
Let $\cM$ be a gridding matrix such that every entry of $\cM$ has bounded grid-width. Then one of 
the following holds:
\begin{itemize}
\item Either $G_\cM$ is a forest that avoids a bumper-ended path, $\Grid(\cM)$ has bounded 
grid-width and \PPPM{$\Grid(\cM)$} can be decided in polynomial time, or
\item $G_\cM$ contains a bumper-ended path or a cycle, $\Grid(\cM)$ has unbounded grid-width and 
\PPPM{$\Grid(\cM)$} is NP-complete.
\end{itemize}
\end{theorem}

Unlike monotone grid classes, the gridding matrices of general grid classes may contain entries 
corresponding to finite nonempty classes. However, we can ignore these entries without affecting the 
properties we are interested in. To see this, let $\cM'$ be the gridding matrix obtained by removing 
all finite entries from a gridding matrix~$\cM$. Note that the cell graph of $\cM$ is equal to the 
cell graph of~$\cM'$. Moreover, the NP-completeness of \PPPM{$\Grid(\cM')$} trivially implies the 
NP-completeness of \PPPM{$\Grid(\cM)$}. Finally, $\Grid(\cM')$ has bounded grid-width if and only if 
$\Grid(\cM)$ has bounded grid-width since inserting a constant number of points into a permutation 
increases its grid-width at most by a constant. Thus, if $\cM'$ satisfies one of the two options in 
Theorem~\ref{thm:general-grid}, then $\cM$ satisfies this option as well. From now on, we will 
assume that $\cM$ contains only infinite (or empty) entries.

%\begin{lemma}
%Let $\cM'$ be the gridding matrix obtained from a gridding matrix $\cM$ by replacing every finite 
%class with an empty class. The \{grid, path\}-width of $\Grid(\cM)$ is bounded if and only if the 
%\{grid-path\}-width of $\Grid(\cM')$ is bounded. Moreover, if \PPPM{$\Grid(\cM')$} is 
%NP-complete then so is \PPPM{$\Grid(\cM)$}.
%\end{lemma}
%TODO appendix!
%\begin{proof}
%Suppose that $\sigma$ is a permutation created by adding a new point $p$ arbitrarily to a 
%permutation $\pi$. We claim that $\gw(\sigma) \leq \gw(\pi) +1$ and $\pw(\sigma) \leq \pw(\pi) 
%+1$. 
%Let $T$ be the optimal (caterpillar) grid tree for $\pi$. We define a (caterpillar) grid tree $T'$ 
%of $\sigma$ by taking a root vertex with left child being the point $p$ and right child the tree 
%$T$. The grid-complexity of any vertex in $T$ has increased by at most 1.
%
%An $\cM$-gridded permutation $\pi$ contains at most $C$ points in the cells corresponding to the 
%finite classes for some constant $C$. Thus if \{grid, path\}-width of $\Grid(\cM')$ is bounded by 
%$k$ then \{grid, path\}-width of $\Grid(\cM')$ is bounded by $k + C$. And finally since 
%$\Grid(\cM')$ is a subclass of $\Grid(\cM)$, the NP-completeness of $\Grid(\cM')$-Pattern PPM 
%implies the NP-completeness of $\Grid(\cM)$-Pattern PPM.
%\end{proof}

One natural 
way to generalize the notion of monotone classes is to consider classes that have bounded 
path-width in left-to-right ordering or 
in the bottom-to-top ordering. For permutation $\pi$, the \emph{horizontal path-width} is
$\pw^\sigma(\pi)$ where $\sigma_i = i$, and 
the \emph{vertical path-width} is $\pw^\sigma(\pi)$ where $\sigma_i = 
\pi^{-1}_i$.
The horizontal path-width was introduced independently by Ahal and Rabinovich~\cite{Ahal2000} and 
Albert et al.~\cite{Albert2001} in the context of designing permutation pattern matching 
algorithms. Moreover, there is a connection to the so-called 
insertion-encodable classes which appear often in the area of permutation classes enumeration, see 
e.g.~\cite{Albert2005,Bean2019,Vatter2012}. 
%In particular, regular 
%insertion-encodable classes are precisely the classes with bounded  vertical path-width which are 
%finitely based. 

A \emph{horizontal monotone juxtaposition} is a monotone grid class $\Grid(\cC\;\cD)$ where both 
$\cC$ and $\cD$ are non-empty. Similarly, a \emph{vertical monotone juxtaposition} is a monotone 
grid class $\Grid\left(\begin{smallmatrix}\cC \\ \cD\end{smallmatrix}\right)$. The following two 
lemmas are stated for the horizontal path-width, but their symmetric versions hold for the vertical 
path-width. The next lemma was also proved, in a different form, by Albert et al.~\cite{Albert2005} 
in the context of regular insertion encodings.
\begin{lemma}
\label{lem:LR-pw}
For a permutation class $\cC$ the following are equivalent:
\begin{enumerate}[(a)]
\item $\cC$ has unbounded horizontal path-width,
\item $\cC$ contains arbitrarily large horizontal alternations, and 
\item $\cC$ contains a horizontal monotone juxtaposition as a subclass.
\end{enumerate}
\end{lemma}
\begin{proof}
Suppose (a) holds and for every integer $k$ there is a permutation $\pi^k \in \cC$ such that the
horizontal path-width of $\pi^k$ is at least $k$. Then there is $i$ such that the set $\{\pi^k_1, 
\dots, \pi^k_i\}$ has intervalicity at least $k$. Each pair of neighboring intervals is separated 
by $\pi^k_j$ for some $j > i$. Therefore, $\pi^k$ contains a horizontal 
alternation of size at least $2k-1$ which proves (b). On the other hand, a horizontal alternation 
of size $2k$ must have a horizontal path-width at least $k$ and thus (b) implies (a).
	
Now suppose that (b) holds. A \emph{monotone horizontal alternation} is a horizontal alternation 
whose set of odd entries and set of even entries both form monotone sequences. We claim that every 
horizontal alternation $\pi$ of size $2k^4$ contains a monotone horizontal alternation of size 
$2k$. By applying the Erdős–Szekeres theorem~\cite{Erdos1935} on the odd entries of $\pi$ we obtain 
a horizontal alternation $\pi'$ of size at least $2k^2$ whose odd entries form a monotone sequence. 
Applying the Erdős–Szekeres theorem again on the even entries of $\pi'$ yields a monotone 
horizontal alternation $\pi''$ of size at least $2k$. Therefore, $\cC$ contains arbitrarily large 
monotone horizontal alternations. There are only four possible types of such alternations depending 
on the type of the monotone sequences. Therefore, $\cC$ also contains arbitrarily large 
alternations belonging to a horizontal monotone juxtaposition $\Grid(\cD_1\;\cD_2)$ for some choice 
of $\cD_1,\cD_2\in\{\Inc,\Dec\}$. Since every $\sigma \in \Grid(\cD_1\;\cD_2)$ is contained in a 
sufficiently large monotone alternation, in fact $\cC$ must contain the whole class 
$\Grid(\cD_1\;\cD_2)$ as a subclass.
	
On the other hand, if (c) holds then $\cC$ contains arbitrarily large monotone horizontal 
alternations which trivially implies (b).
\end{proof}

%Whenever we select points from a bounded number of horizontal intervals in a 
%permutation with bounded horizontal path-width, these points also span only a 
%bounded number of vertical intervals.

\begin{lemma}
\label{lem:hpw-propgation}
Let $\pi$ be a permutation from a class $\cC$ with bounded horizontal path-width and let 
$S$ be a subset of $\pi$ such that $\intx(S)=  k$. Then $\inty(S) \leq \alpha k$ 
where the constant $\alpha$ depends only on $\cC$.
\end{lemma}

\begin{proof}
By Lemma~\ref{lem:LR-pw}, there exists an $l$ such that $\cC$ does not contain any vertical 
alternation of size $l$. Let $\cI$ be the interval family of size $k$ such that $\bigcup \cI = 
\Pi_x(S)$ and let $I$ be an interval of~$\cI$. Let $S_I$ be the subset of $S$ such 
that $\Pi_x(S_I) = I$ and let $\cJ$ be the smallest interval family such that $\Pi_y(S_I) = \bigcup 
\cJ$.
We claim that $\cJ$ contains at most $2l-1$ intervals. For contradiction, suppose that the size of 
$\cJ$ is at least $2l$. Then between each pair of consecutive intervals of $\cJ$ there is a value 
$j$ such that $\pi^{-1}_j$ lies outside the interval $I$. There is at least $2l -1$ gaps between 
intervals of $\cJ$ and therefore by the pigeon-hole principle either $l$ of them contain a point to 
the right of $I$ or at least $l$ of the gaps contain a point to the left of $I$. Either way, we 
obtain a horizontal alternation of size $l$, which is a contradiction.

For each interval $I \in\cI$, we showed that the intervalicity of $\Pi_y(S_I)$ is at most $2l -1$ 
and thus the  intervalicity of $\Pi_y(S)$ is at most $k(2l-1)$.
\end{proof}

An ordered pair $(p,q)$ of vertices in $G_\cM$ is a \emph{bumper} if either $\cM_{q}$ has unbounded 
horizontal path-width and shares the same column with $\cM_{p}$, or if $\cM_{q}$ has unbounded 
vertical path-width and shares the same row with  $\cM_{p}$. A \emph{bumper-ended path} is a path 
$P =  p_1, \ldots, p_k$ in $G_\cM$ such that both $(p_2,p_1)$ and $(p_{k-1},p_k)$ are bumpers.

\begin{figure}
\centering
\raisebox{-0.5\height}{\includegraphics[scale=0.4]{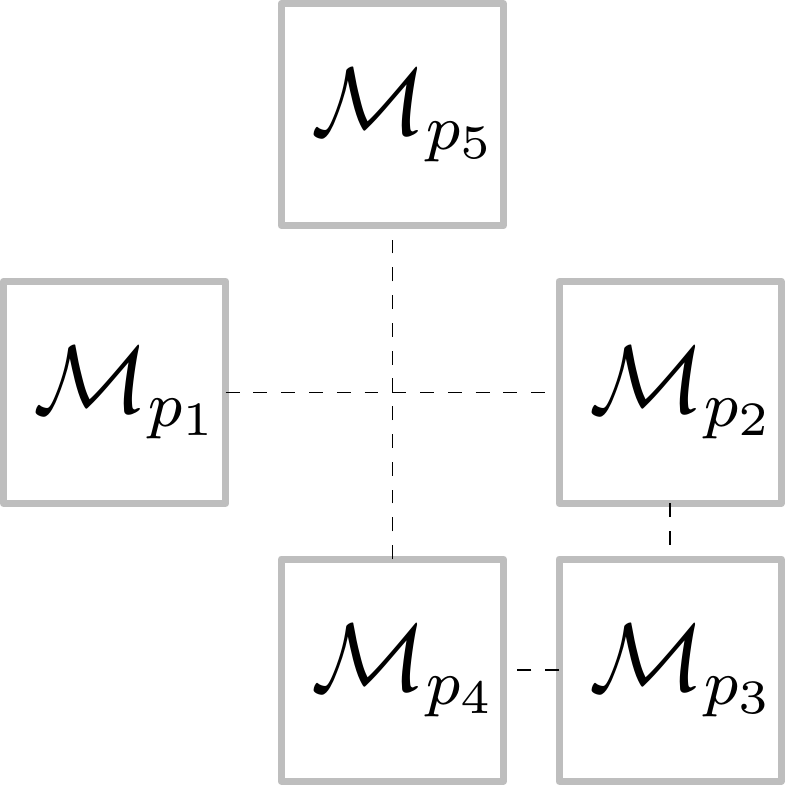}}
\hspace{0.05in}
{\huge$\rightarrow$}
\hspace{0.05in}
\raisebox{-0.5\height}{\includegraphics[scale=0.4]{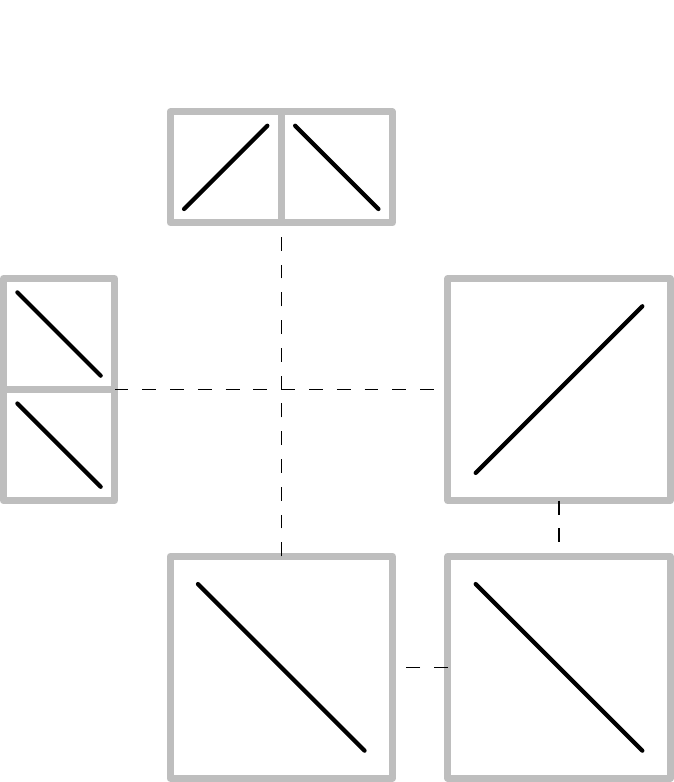}}
\hspace{0.05in}
{\huge$\rightarrow$}
\hspace{0.05in}
\raisebox{-0.5\height}{\includegraphics[scale=0.4]{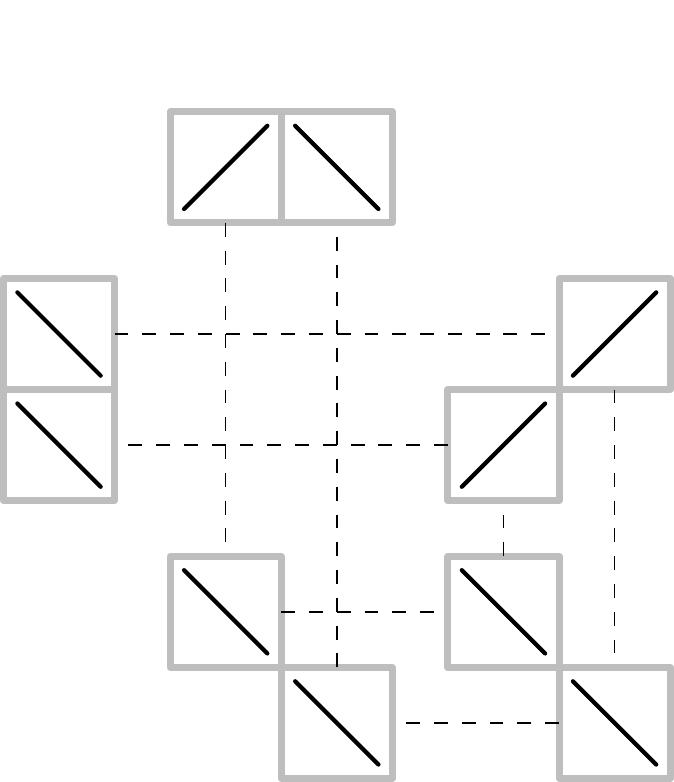}}
\caption{Transforming a bumper-ended path to cycle in the proof of Lemma~\ref{lem:bumper-path}.}
\label{fig:bumper-to-cycle}
\end{figure}

\begin{lemma}
\label{lem:bumper-path}
If $G_\cM$ contains a bumper-ended path then $\Grid(\cM)$ has unbounded grid-width and the 
problem \PPPM{$\Grid(\cM)$} is NP-complete.
\end{lemma}
\begin{proof}
We aim to show that  $\Grid(\cM)$ must contain a cyclic monotone grid class as its subclass. The 
proof is illustrated in Figure~\ref{fig:bumper-to-cycle}.
Consider the bumper-ended path $p_1, p_2, \ldots, p_k$. Let us assume that both $M_{p_1}$ and 
$\cM_{p_k}$ have unbounded horizontal path-width as the other cases can be proved in an analogous 
way. Each of the infinite classes $\cM_{p_i}$ contains a monotone subclass $\cC_i$ due to the 
Erdős–Szekeres theorem~\cite{Erdos1935}. Moreover, the classes $\cM_{p_1}$ and $\cM_{p_k}$ 
contain a monotone juxtaposition by Lemma~\ref{lem:LR-pw}.
Let $\Grid(\cC_1\;\cD_1)$ be the juxtaposition contained in $\cM_{p_1}$ and $\Grid(\cC_k \; 
\cD_k)$ the juxtaposition contained in $\cM_{p_k}$. We define the monotone gridding matrix 
$\cM'$ by replacing every entry of $\cM$ with the following $2 \times 2$ matrix:
\begin{itemize}
\item entry $\cM_{p_i}$ for $i$ between 2 and $k-1$ is replaced with $\cC_i^{\times2}$ 
\item entry $\cM_{p_t}$ for $t \in \{1,k\}$ is replaced with 
$\left(\begin{smallmatrix}\emptyset & \emptyset\\ 
\cC_t &\cD_t\end{smallmatrix}\right)$, and
\item every other entry is replaced with an empty $2\times2$ matrix.
\end{itemize}

Clearly, $\Grid(\cM')$ is still a subclass of $\Grid(\cM)$. The $(i,j)$-block of $\cM'$ is the 
$2\times2$ submatrix obtained from the $(i,j)$-cell in $\cM$. If we forget about the blocks of 
$p_1$ and $p_k$ we are left with two disjoint copies of the original path. Adding back the blocks 
connects the endpoints of both paths together and creates a cycle since $p_2$ shares the same 
column with $p_1$ and $p_{k-1}$ shares the same column with~$p_{k}$. Thus, $\Grid(\cM')$ is a 
monotone grid subclass of $\Grid(\cM)$ whose cell graph contains a cycle and the claim follows from 
Proposition~\ref{prop:cyclic-grid}.
\end{proof}

\begin{proposition}
\label{prop:bumper-bounded} Let $\cM$ be a gridding matrix such that every entry of $\cM$ has bounded 
grid-width, and there are no nonempty finite entries. If $G_\cM$ is a forest that avoids a 
bumper-ended path then $\Grid(\cM)$ has bounded grid-width and \PPPM{$\Grid(\cM)$} can be decided 
in polynomial time. 
\end{proposition}
\begin{proof}
First, suppose that $G_\cM$ contains more than one component. In that case, choose a component 
of $G_\cM$, and let $\cM_1$ be the submatrix of $\cM$ spanned by the rows and columns containing 
the vertices of the chosen component, while $\cM_2$ is the submatrix spanned by the remaining 
rows and colums of~$\cM$. An $\cM$-gridded permutation 
$\pi$ can be partitioned into two subpermutations $\pi_1$ and $\pi_2$ where $\pi_i$ is the 
$\cM_i$-gridded subpermutation of $\pi$ consisting of the rows and columns of~$\cM_i$. Let $T_i$ be 
the optimal grid tree of~$\pi_i$. 
We define a grid tree $T$ of $\pi$ by taking a root vertex with children $T_1$ and $T_2$. The 
grid-complexity of any vertex in $T_1$ or $T_2$ has increased at most by $\max(k,\ell)$ where 
$k$ and $\ell$ are the dimensions of $\cM$. Therefore $\gw(\pi) \le \max(\gw(\pi_1), 
\gw(\pi_2))+\max(k,\ell)$. Applying this argument inductively shows that $\Grid(\cM)$ has bounded 
grid-width if and only if the grid-width of $\Grid(\cM')$ is bounded for every submatrix $\cM'$ of 
$\cM$ spanned by a connected component $G_\cM$. In the rest of the proof, we assume that 
$G_\cM$ is a tree. The proof is  based on a sequence of claims that will be stated and proven 
independently.

\begin{claim}\label{cla-r}
The tree $G_\cM$ contains a vertex $r$ such that for any other vertex $q\neq r$, the path from $r$ 
to $q$ does not end in a bumper. 
\end{claim}
Assume for contradiction that the claim fails. Let $r$ be any vertex of $G_\cM$. By assumption, 
there is a vertex $q\neq r$ such that the path from $r$ to $q$ ends in a bumper $(p,q)$. Choose 
such a vertex $q$ as far as possible from~$r$. Applying our assumption for $q$ in the role of $r$, 
there is a vertex $q'\neq q$ such that the path from $q$ to $q'$ ends in a bumper $(p',q')$. If the 
path from $q$ to $q'$ contains the vertex $p$, then it is a bumper-ended path, which is impossible. 
If the path from $q$ to $q'$ avoids $p$, it means that the path from $r$ to $q'$ ends in the bumper 
$(p',q')$ and is strictly longer than the path from $r$ to $q$, which contradicts the choice 
of~$q$. This proves Claim~\ref{cla-r}.

\begin{figure}
\centerline{\includegraphics[width=0.7\textwidth]{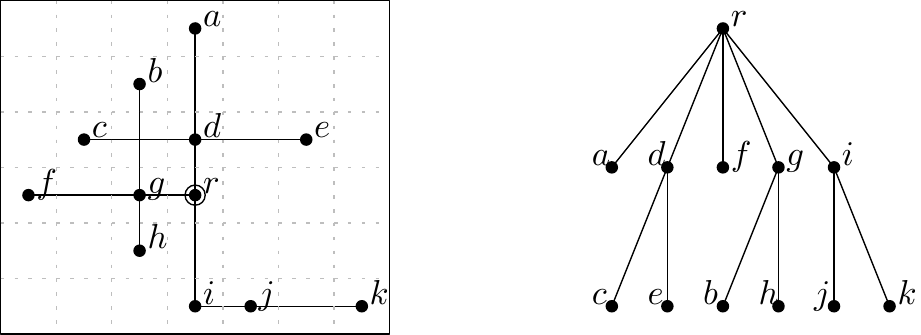}}
 \caption{Example of a gridding matrix $\cM$ with the tree $G_\cM$ (left) and the corresponding 
rooted tree $T_\cM$ (right).}\label{fig-tm}
\end{figure}

Let $r$ be the vertex of $G_\cM$ whose existence is guaranteed by Claim~\ref{cla-r}. We now define 
a rooted tree $T_\cM$ on the same vertex set as $G_\cM$ as follows (see Figure~\ref{fig-tm}). The 
vertex $r$ is the root $T_\cM$. For a vertex $v \neq r$ in $G_\cM$, we set the parent of $v$ in 
$T_\cM$ to be the furthest vertex on the $vr$-path in $G_\cM$ that shares the same row or column  
with $v$. Observe that whenever a vertex $v$ shares the same column with its parent $w$ in $T_\cM$ 
then the entry $\cM_v$ has bounded horizontal path-width, and whenever $v$ shares the same row with 
$w$, the entry $\cM_v$ has bounded vertical path-width. The \emph{dominant cell} of a row, or a 
column, is the cell $v$ such that all the other cells in the row, or column, are its children 
in~$T_\cM$.

Let $\pi$ be an $\cM$-gridded permutation, and let $\pi_v$ denote the subset of points contained in 
a cell~$v$. Assume that $\pi_v$ is nonempty whenever $\cM_v$ is a nonempty cell of~$\cM$. We define 
an auxiliary directed graph $G_\pi$ on the points of $\pi$ whose every connected component is a 
tree rooted in some point of~$\pi_r$. Suppose that the vertex $v$ of $T_\cM$ shares the same 
column with its parent~$w$. The parent of a point $p$ in $\pi_v$ is the nearest point in 
$\pi_w$ to the right of $p$, and if there is no such point ($p$ lies to the right of all the 
points of~$\pi_w$) then the rightmost point in~$\pi_w$. If $v$ and $w$ share the same 
row, then the parent of $p$ is the nearest point in $\pi_w$ above $p$, or if there is no such 
point ($p$ lies above all the points of $\pi_w$) then the topmost point in~$\pi_w$. Let 
$P$ be a subset of the permutation diagram of~$\pi$. The point set $\overline{P}$ contains $P$ and 
every point that lies in $G_\pi$ in a subtree of some point $p \in P$.

Let $v$ be a non-empty cell such that $\pi_v$ contains $m$ points. The permutation 
$\widetilde{\pi_v}$ is 
the standardized version of $\pi_v$, i.e. the point set inside $[m]\times[m]$ that is isomorphic to 
$\pi_v$. The construction of the graph $G_\pi$ guarantees the following property and its symmetric 
version.

\begin{observation}
\label{obs:cons-int-a} 
Let $v$ be the dominant cell of its row. Let $S$ be a subset of $\pi_v$, let $\widetilde{S}$ be the 
corresponding subset of the standardized $\widetilde{\pi_v}$ and let $S'$ be the set containing $S$ 
and all
its children in $G_\pi$ that lie in the same row. Then $\inty(\widetilde{S}) = 
\inty(S')$. 
Symmetrically, if $v$ is dominant in its column, $S$ and $\widetilde{S}$ are as above, and $S'$ 
contains $S$ 
and all its children in $G_\pi$ in the same column, then $\intx(\widetilde{S}) = 
\intx(S')$. 
\end{observation}

We inductively define a function $h$ on the vertex set of $T_\cM$, which will later serve us as an 
upper bound for the grid-width of any $\pi \in \Grid(\cM)$. For any leaf $u$ of 
$T_\cM$ we set $h(u)=1$. For any other vertex $v$, we let $W$ be the set of children of $v$ and 
define $h(v) = 1 + 
\sum_{w\in W} \alpha_{w}h(w)$ where $\alpha_{w}$ is the constant obtained as follows. If $v$ 
shares a column with $w$, then $\alpha_{w}$ is the constant from Lemma~\ref{lem:hpw-propgation} 
applied on the class $\cM_{w}$, otherwise it is the constant from the `vertical' version of 
Lemma~\ref{lem:hpw-propgation} applied on the class $\cM_{w}$. We state only one of the symmetric 
versions of the following two claims. However, we are proving both of them simultaneously by 
induction.

\begin{claim}
\label{claim:bounded-int-a}
Let $S$ be a subset of the $i$-th column of the $\cM$-gridding of $\pi$ such that $\intx(S)=1$.   
Let $v$ be the dominant cell of the $i$-th column and let $S_v = S\cap \pi_v$. Then 
$\intx(\overline{(S \setminus S_v)}\cup S_v) \leq h(v)$ 
and $\inty(\overline{S\setminus S_v}) \leq h(v) - 1$.
\end{claim}
We remark that if $v$ is not equal to the root $r$, then the set $\overline{(S \setminus S_v)}\cup 
S_v$ from the claim above is actually equal to~$\overline{S}$.

To prove Claim~\ref{claim:bounded-int-a}, suppose first that $v$ is the only nonempty cell in its 
column, and therefore $v$ is a leaf of $T_\cM$. Then $S=S_v$, and hence
$\intx(\overline{(S\setminus S_v)}\cup S_v)=\intx(S) = 1=h(v)$ and 
$\inty(\overline{S\setminus S_v})=\inty(\emptyset) = 0=h(v)-1$, as claimed.

Now suppose that $v$ is not the only nonempty cell in its column, and let $C$ be the set of 
nonempty cells different from $v$ in the same column as~$v$. Note that each cell in $C$ is
a child of $v$ in $T_\cM$, and if $v\neq r$, then $C$ is precisely the set of children 
of~$v$. Observe that $\overline{S\setminus S_v}$ is a disjoint union of the sets $\overline{S_w}$ 
over all $w\in C$. For a cell $w\in C$, let $S_w$ be the set $S\cap\pi_w$, let 
$\widetilde{\pi_w}$ be 
the standardization of $\pi_w$, and let $\widetilde{S_w}$ be the subset of $\widetilde{\pi_w}$ that 
corresponds to~$S_w$.  
From Lemma~\ref{lem:hpw-propgation} we get $\inty(\widetilde{S_w})\le \alpha_w$, and in 
particular, 
$\widetilde{S_w}$ can be partitioned into sets $\widetilde{S_w^1}, 
\widetilde{S^2_w},\dotsc,\widetilde{S^{\ell_w}_w}$ for some $\ell_w\le 
\alpha_w$, where $\inty(\widetilde{S^i_w})=1$ for each~$i$. Let $S^i_w$ be the 
subset of 
$\pi_w$ that corresponds to $\widetilde{S_w^i}$ in the standardization~$\widetilde{\pi_w}$. Let 
$R_w^i$ be the set 
$S_w^i$ together with all its children in $G_\pi$ that lie in the same row. By 
Observation~\ref{obs:cons-int-a}, for each $i$ we have $\inty(R_w^i)=1$. Using the symmetric version 
of Claim~\ref{claim:bounded-int-a} with each $R_w^i$ in 
the role of $S$ shows that

\begin{align*}
 \inty(\overline{S\setminus S_v})&\le \sum_{w\in C} \inty(\overline{S_w})\\
&\le \sum_{w\in C} \sum_{i=1}^{\ell_w} \inty(\overline{S_w^i})\\
&= \sum_{w\in C} \sum_{i=1}^{\ell_w} \inty\left(\overline{(R_w^i\setminus 
S_w^i)}\cup S_w^i\right)\\
&\le \sum_{w\in C} \alpha_w h(w)\\
&\le h(v)-1,
\end{align*}
and similarly,
\begin{align*}
\intx(\overline{(S \setminus S_v)}\cup S_v) 
&=\intx\left(S\cup\bigcup_{w\in C} \overline{S_w}\right)\\
&=\intx\left(S\cup\bigcup_{w\in C}\bigcup_{i=1}^{\ell_w} \overline{R_w^i\setminus 
S_w^i}\right)\\
&\le 1+\sum_{w\in C}\sum_{i=1}^{\ell_w}\intx(\overline{R_w^i\setminus S_w^i})\\
&\le 1+\sum_{w\in C}\alpha_w (h(w)-1)\\
&\le h(v),
\end{align*}
proving Claim~\ref{claim:bounded-int-a}.

We will now define, for every $p\in \pi$, a grid tree $T_p$ whose leaves are exactly the points 
in~$\overline{\{p\}}$. The definition proceeds inductively on the size of~$\overline{\{p\}}$. If $p$ 
has no children in~$G_\pi$, that is if $\overline{\{p\}} =\{p\}$, then $T_p$ consists of the single 
vertex~$p$. Suppose now that $p$ has at least one child in~$G_\pi$. Recall that each child of $p$ 
belongs to a cell in the gridding of $\pi$ which is in the same row or in the same column as the cell 
of~$p$. Let $C$ and $R$ denote, respectively, the set of children of $p$ in the same column and the 
set of children of $p$ in the same row. Note that $C$ and $R$ are disjoint, and if $p$ does not 
belong to the root cell $\pi_r$ then one of $C$ and $R$ is empty.

Recall that a caterpillar is a binary tree whose every internal node has at least one leaf child. 
Note that the leaves of a caterpillar can be ordered top to bottom by their distance from the root, 
where the order of the bottommost pair of leaves is irrelevant. 

If $C$ is nonempty, we construct a tree $T^C_p$ in the following two steps:
\begin{itemize}
\item Construct a caterpillar whose leaves are the points from $C\cup\{p\}$, and the top-to-bottom 
order of the leaves in the caterpillar coincides with the left-to-right order of the points in~$\pi$.
\item In the caterpillar constructed above, for each $q\in C$ replace the leaf $q$ with a copy of 
the tree~$T_q$. Call the resulting tree $T^C_p$.
\end{itemize}

Symmetrically, if $R$ is nonempty, construct a tree $T^R_p$ by first taking the caterpillar whose 
leaves in top-to-bottom order are the points of $R\cup\{p\}$ in top-to-bottom order, and then for 
each $q\in R$ replace the leaf $q$ with a copy of~$T_q$.

If the set $R$ is empty, we define $T_p=T^C_p$, and if $C$ is empty, we define $T_p=T^R_p$. If 
both $C$ and $R$ are nonempty (which may only happen when $p$ is in $\pi_r$), we let $T_p$ be the 
tree obtained by replacing the leaf $p$ in $T^C_p$ by a copy of~$T^R_p$. Note that in all the cases, 
the leaves of $T_p$ form precisely the set~$\overline{\{p\}}$.

\begin{claim}\label{cla-tree}
Let $v$ be a nonempty cell of~$\cM$. If $v\neq r$, then for every $p\in\pi_v$, the tree $T_p$ has 
grid-width at most~$h(v)$. For every $p\in r$, the tree $T_p$ has grid-width at most~$2h(r)$.
\end{claim}

We prove the claim by induction on the size of~$T_p$. The claim clearly holds when $T_p$ is the 
single vertex~$p$. Suppose now that $T_p$ has more vertices, and that $v\neq r$. In such case $T_p$ 
is equal to $T^C_p$ or to $T^R_p$. Suppose that $T_p=T^C_p$, the other case being symmetric. Let $C$ 
be again the set of children of $p$ in $G_\pi$ (necessarily, they are all in the same column of the 
gridding as the point $p$, since $v$ is not the root vertex). Let $u$ be a node of $T_p$, and let 
$L_u$ be the set of leaves of the subtree of $T_p$ rooted at~$u$. Our goal is to show that the grid 
complexity of $L_u$ is at most~$h(v)$. If $u$ is the leaf $p$ or $u$ is inside a copy of $T_q$ for 
some $q\in C$, the claim follows by induction. Suppose that $u$ is a node of the caterpillar from 
which $T_p$ was constructed. Let $S$ be the set of points in $L_u$ that are in the same gridding 
column as~$p$. By the construction of $T_p$, the set $S$ satisfies $\intx(S)=1$, and
$L_u$ is equal to~$\overline{S}$. Note that the set $S_v=S\cap \pi_v$ is either empty or contains 
the single point $p$. By Claim~\ref{claim:bounded-int-a}, 
\[
\intx(L_u)=\intx(\overline{S})=\intx(\overline{(S\setminus S_v)}\cup S_v)\le h(v)
\]
and
\[
\inty(L_u)\le \inty(S_v)+\inty(\overline{S\setminus S_v})\le \inty(\{p\})+\inty(\overline{S\setminus 
S_v})\le 1+(h(v)-1)=h(v).
\]
This shows that $L_u$ has grid complexity at most $h(v)$, and therefore $T_p$ has grid-width at most 
$h(v)$.

It remains to deal with the case when $p$ belongs to the root cell~$r$. Using the same argument as 
in the first part of the proof, we again see that both $T_p^C$ and $T_p^R$ have grid-width at 
most~$h(r)$. Moreover, for each node $u$ of $T_p$ the subtree of $T_p$ rooted at $u$ is either equal 
to a subtree of $T_p^C$, or it is equal to a subtree of $T_p^R$, or it contains the entire tree 
$T_p^R$ together with a subtree of $T_p^C$. In the former two cases, the set of leaves of the subtree 
has grid complexity at most $h(r)$, in the last case it has grid complexity at most $2h(r)$. This 
proves Claim~\ref{cla-tree}.

We are ready to construct a grid tree $T$ of the permutation~$\pi$ and provide a bound on its 
grid-width. By assumption, the entries of $\cM$ have bounded grid-width, and we let $g$ be the 
grid-width of the root entry~$\cM_r$. Let $\pi^*_r$ be the standardization of $\pi_r$, and let $T_r$ 
be the optimum grid tree of $\pi^*_r$; in particular, $T_r$ has grid-width at most~$g$. A grid tree 
$T$ of the whole permutation $\pi$ is obtained by taking $T_r$ and replacing the leaf corresponding 
to a point $p\in\pi_r$ with the tree~$T_p$. We claim that $T$ has grid-width at most $4gh(r)$. The 
tree $T$ contains every point of~$\pi$, and we showed in Claim~\ref{cla-tree} that the grid-width of 
any node contained in a copy of some $T_p$ is at most~$2h(r)$. 

Let now $u$ be a node of $T$ that is not contained in any copy of the tree $T_p$, in other words, 
$u$ is an internal node of~$T_r$. Let $L^*\subseteq\pi^*_r$ be the set of leaves of $T_r$ in 
the subtree rooted at~$u$, and let $L$ be the subset of $\pi_r$ that is mapped to $L^*$ by the 
standardization that maps $\pi_r$ to~$\pi_r^*$. Then the subset of $\pi$ contained in the subtree of 
$T$ rooted at $u$ is precisely~$\overline{L}$. 

Applying Observation~\ref{obs:cons-int-a}, we see that $L$ together with its neighbors in $G_\pi$ 
spans at most $g$ consecutive intervals in the row and column of the $r$-cell. By applying 
Claim~\ref{claim:bounded-int-a} individually on each of these $2g$ intervals, we get that the 
grid-complexity of $\overline{L}$ is at most $4gh(r)$. It follows that $\Grid(\cM)$ has bounded 
grid-width, and therefore \PPPM{$\Grid(\cM)$} can be decided in polynomial time.
\end{proof}

To complete the proof of Theorem~\ref{thm:general-grid}, it suffices to point out that if $G_\cM$ 
contains a cycle then $\Grid(\cM)$ contains a monotone grid subclass $\Grid(\cM')$ where $\cM'$ is 
obtained by replacing every infinite class in $\cM$ by its monotone subclass. Applying 
Proposition~\ref{prop:cyclic-grid} then wraps up the proof of Theorem~\ref{thm:general-grid}.

\section{Concluding remarks and open problems}

The \textsc{$\cC$-PPM} is the problem of determining whether a permutation $\pi \in \cC$ is 
contained in a permutation $\tau \in \cC$. Even though \PPPM{$\Av(321)$} is NP-complete, 
\textsc{$\Av(321)$-PPM} can be decided in polynomial time~\cite{Guillemot2009}. This leads to the 
natural question of whether the same can happen in the universe of grid classes.

\begin{problem}
Is there any (monotone) gridding matrix $\cM$ such that \PPPM{$\Grid(\cM)$} is NP-complete, 
while \textsc{$\Grid(\cM)$-PPM} can be decided in polynomial time?
\end{problem}

A \emph{path class of order $k$} is a monotone grid class whose cell graph is a path on $k$ 
vertices. Let us say that a permutation class $\cC$ \emph{contains paths of all orders}, if for 
every $k$, the class $\cC$ contains as a subclass a path class of order~$k$. Note that by 
Lemma~\ref{lem:long-path-width}, such a class $\cC$ has unbounded grid-width. In fact, all the 
known examples of classes with unbounded grid-width contain paths of all orders. We may therefore 
ask whether this property precisely characterizes the classes with unbounded grid-width. 

\begin{problem}
Does every class with unbounded grid-width contain paths of all orders?
\end{problem}

The existence of paths of all orders may also help with establishing the NP-completeness of 
\PPPM{$\cC$}. Suppose that $\cC$ is a class that contains paths of all orders. It is not hard to 
argue that such a class $\cC$ necessarily contains, for every $k$, a monotone grid subclass whose 
grid graph is a properly turning path on $k$ vertices. If we additionally assume that for a given 
integer $k$ we are able to construct, in time polynomial in $k$, such a properly turning path class 
of order $k$ contained in $\cC$, then we can adapt the hardness reduction from the proof of 
Proposition~\ref{prop:cyclic-grid} to show that \PPPM{$\cC$} is NP-complete.

The results of Ahal and Rabinovich~\cite{Ahal2000} imply that \PPPM{$\cC$} is polynomial whenever 
$\cC$ has bounded grid-width. On the other hand, in all the known examples of a class $\cC$ with 
unbounded grid-width where the complexity of \PPPM{$\cC$} is known, the \PPPM{$\cC$} problem is 
NP-complete. We wonder whether bounded grid-width might be the property characterizing the 
complexity of~\PPPM{$\cC$}.

\begin{problem}
Is it true that \PPPM{$\cC$} is NP-complete whenever $\cC$ has unbounded grid-width, and polynomial 
otherwise?
\end{problem}

\bibliography{gridwidth}

\end{document}